\documentclass{amsart}
\usepackage{hyperref}
\hypersetup{nesting=true,debug=true,naturalnames=true}
\usepackage{graphicx,amssymb}

\usepackage[square,sort&compress,comma,numbers]{natbib}
\usepackage{nicefrac,xcolor,upref}

\let\<\langle
\let\>\rangle
\usepackage[all,pdf]{xy}
\UseComputerModernTips

\newcommand{\mailurl}[1]{\email{\href{mailto:#1}{#1}}}
\let\uml\"

\title[Inhomogeneous Diophant{i}ne approximat{i}on over power series f{i}elds]
{Inhomogeneous Diophant{i}ne approximat{i}on over f{i}elds of formal power series}

\author{Yann Bugeaud}
\address{\protect{Universit\'{e} de Strasbourg, 7, rue Ren\'{e} Descartes, 67084 Strasbourg, FRANCE}}
\mailurl{bugeaud@math.unistra.fr}

\author{L.~Singhal}
\address{\protect{Beijing Internat{i}onal Center for Mathemat{i}cal Research, Peking University, 100\,871 Beijing, P.~R.~CHINA}}
\mailurl{lsd@bicmr.pku.edu.cn}

\author{Zhenliang Zhang}
\address{\protect{School of Mathemat{i}cal Sciences, Henan Inst{i}tute of Science and Technology, 453\,003 Xinxiang, P.~R.~CHINA}}
\mailurl{zhliang\_zhang@hotmail.com}

\subjclass[2010]{}

\newtheorem{theorem}{Theorem}[section]
\newtheorem{lem}[theorem]{Lemma}
\newtheorem{prop}[theorem]{Proposit{i}on}
\newtheorem{cor}[theorem]{Corollary}

\theoremstyle{definition}
\newtheorem{Def}[theorem]{Definit{i}on}
\newtheorem{notation}[theorem]{Notat{i}on}

\numberwithin{equation}{section}

\newcommand{\abs}[1]{\ensuremath{\left\lvert\,#1\,\right\rvert}}
\newcommand{\braket}[1]{\ensuremath{\left\langle\,#1\,\right\rangle}}

\newcommand{\C}{\ensuremath{\mathbb{C}}}
\newcommand{\eps}{\ensuremath{\varepsilon}}
\newcommand{\F}{\ensuremath{\mathbb{F}}}
\newcommand{\FqT}{\ensuremath{\F_q \left[ T \right]}}
\newcommand{\Fqt}{\ensuremath{K}}
\newcommand{\FqTinv}{\ensuremath{\F_q \left[\left[ T^{-1} \right]\right]}}
\newcommand{\Fqtinv}{\ensuremath{K_{\infty}}}

\newcommand{\hatmu}[2]{\ensuremath{\widehat{\mu}\,\big(\,#1,\,#2\,\big)}}
\newcommand{\hatomega}[1]{\ensuremath{\widehat{\omega} \left( #1 \right)}}
\newcommand{\invT}{\ensuremath{T^{-1}}}

\newcommand{\N}{\ensuremath{\mathbb{N}}}
\newcommand{\norm}[1]{\ensuremath{\left\|\,#1\,\right\|}}
\renewcommand{\P}{\ensuremath{\underline{P}}}
\newcommand{\R}{\ensuremath{\mathbb{R}}}
\newcommand{\SLtwoFqT}{\ensuremath{\operatorname{SL}_2 (\,\FqT\,)}}
\newcommand{\thetabar}{\ensuremath{\underline{\theta}}}
\newcommand{\tp}[1]{\ensuremath{{}^t #1}}

\newcommand{\vol}[1]{\ensuremath{\nu \left( #1 \right)}}
\newcommand{\x}{\ensuremath{\underline{x}}}
\newcommand{\y}{\ensuremath{\underline{y}}}
\newcommand{\Z}{\ensuremath{\mathbb{Z}}}
\newcommand{\zero}{\ensuremath{\underline{0}}}

\begin{document}

\begin{abstract}
We prove a sharp analogue of Minkowski's inhomogeneous approximation theorem
over fields of power series $\mathbb{F}_q((T^{-1}))$.
Furthermore, we study the approximation to a 
given point $\underline{y}$ in $\mathbb{F}_q((T^{-1}))^2$ 
by the $SL_2(\mathbb{F}_q[T])$-orbit of a given point 
 $\underline{x}$ in $\mathbb{F}_q((T^{-1}))^2$. 
\end{abstract}

\maketitle
%\tableofcontents

  \section{\protect{Introduct{i}on}}\label{S:intro}

By using geometry of numbers,
Minkowski, improving an earlier result of Tsche\-bychev, 
established that for any irrational numbers $\theta$ and any real number 
$\alpha$ not belong to $\mathbb{Z}\theta+\mathbb{Z}$,
there exist infinitely many pairs of integer $(p,q)$,
with $q\neq0$, such that
$$
\abs{q\theta-\alpha-p}\leq \frac{1}{4\abs{q}},
$$
see \cite{Cas57} for the details. The value $1/4$ is best possible. 
Recently, \citeauthor{LN12}~\cite{LN12} obtained an analogous result 
for the orbit of ${\rm SL}(2, \Z) {\xi \choose 1}$ in $\R^2$. 
The purpose of the present paper is to establish analogues 
in the setting of formal power series of these two results 
in inhomogeneous approximation.

Let $q$ be a prime power and $\F_q$ the finite field of order $q$.
Recall that $\F_q[T]$ and $K = \F_q (T)$ 
denote the ring of polynomials and the field of rational functions over $\F_q$, respectively.
Let \Fqtinv\ $=\F_q((T^{-1}))$ denote the field of formal power series 
$x=\sum_{i=-n}^\infty a_{i} T^{-i}$ over the field $\F_q$. 
We equip $\mathbb{F}_q((T^{-1}))$ with the norm $\|x\|=q^n$,
where $a_{-n} \neq 0$ is the first non-zero coefficient in the expansion of the
non-zero power series $x$. This integer $n$ is called the {{degree}} of $x$ and denoted by $\deg x$.

As \Fqtinv\ is a locally compact group under addit{i}on,
it comes with a Haar measure $\nu$ def{i}ned upto mult{i}plicat{i}on by a posi{t}ive constant.
We normalize so that $\vol{\invT \FqTinv} = 1$.
Abusing notat{i}on, we also use $\nu$ to denote the $n$-fold product measure 
on $\Fqtinv^n$ for all $n \geq 1$.
The `integral part' $[ x ]$ of any element $x$ in $\Fqtinv$ stands for the unique polynomial $P$
for which $x - P$ in $\invT\FqTinv$ and $\braket{x}$ refers to $x - [x]$ 
with its $q$-adic norm denoted as
    \begin{equation}\label{E:normx}
      \norm{x} := \abs{\braket{x}}.
    \end{equation}
The norm $\abs{\thetabar}$ of any $\thetabar = \tp{( \theta_1, \ldots, \theta_n )} \in \Fqtinv^n$
equals $\max_{i} q^{\deg \theta_i}$ and its supremum distance from the nearest $\P$ in 
$\FqT^n$ is denoted by
$\norm{\thetabar}_{\Fqtinv}$.
The subscript \Fqtinv\ in the norm expressions will be hidden from now on and
we hope it will be clear from the context as to which norm is being referred to.

Over the fields of formal power series,
there are many results concerning the metrical properties 
for the inhomogeneous diophantine approximation sets,
such as \cite{MS08,Fuc10,Kri11,KN11},
but it seems that the analogue of Minkowski's inhomogeneous 
approximation theorem has not yet been published. The following theorem can also be seen as the inhomogeneous 
version of Dirichlet's Theorem, which says
that for every power series $\xi$ in $\Fqtinv \setminus \Fqt$,
there exist inf{i}nitely many polynomials $Q, P \in \F_q [T]$ such that
    \begin{equation}\label{E:Dirichlet}
      \norm{ Q \xi } \leq \dfrac{1}{q \abs{Q}},
    \end{equation}
see Lemma \ref{L:legendre} below.
\begin{theorem}\label{Th:inhom}
If $\xi \in \Fqtinv \setminus \Fqt$ and $\alpha \notin \FqT + \FqT \xi$,
 then there exist inf{i}nitely many $Q$ in $\FqT$ such that
      \[
        \norm{Q\xi - \alpha} \leq \dfrac{1}{q^2 \abs{Q}}.
      \]
In addition, the factor $q^2$ is best possible. 
Namely, there exists $\xi$ in $\invT\FqTinv \setminus \Fqt$ for which 
\[
  \norm{Q\xi -  \invT ( 1 - \xi ) } \geq \dfrac{1}{q^2 \abs{Q}},
\]
whenever $Q \in \FqT \setminus \{ 0 \}$.  
\end{theorem}
The second part of the above theorem is the power series analogue of Theorem IIA of ~\citep[pp.~48--51]{Cas57}. A simple corollary is that the set $\{ \braket{Q\xi} \mid Q \in \FqT \}$ is dense in $\invT\FqTinv$.
This is also true for approximat{i}ng any vector $\underline{\xi} \in \Fqtinv^m$ by elements
belonging to the subgroup $A\FqT^{n} + \FqT^{m}$ for generic $m \times n$ matrices $A$ and follows from~\cite[Theorem~1.1]{BZ19}. On a slightly di{f}ferent note, we can show that density in the one-dimensional unit ball is achieved by only taking fract{i}onal parts of monic polynomial mult{i}ples of $\xi$ as opposed to all polynomial mult{i}ples.
  \begin{prop}\label{P:monic}
    Let $\xi \in \Fqtinv \setminus \Fqt$ and $\alpha$ be arbitrary. Then, there exist inf{i}nitely many monic polynomials $Q$ such that
    \[
      \abs{Q\xi - \alpha - P} \leq \frac{1}{q\abs{Q}}
    \]
    for some $P$ in \FqT.
  \end{prop}

Following \citep{BL05,BZ19},
we introduce several exponents of homogeneous and inhomogeneous Diophantine approximation.
Let $n$ and $m$ be positive integers and $A$ a matrix in $\mathcal{M}_{n,m}(\Fqtinv)$.
Let $\underline{\theta}$ be in $\Fqtinv^n$.
We denote by $\omega(A,\underline{\theta})$ the supremum of the real numbers $\omega$ for which,
for arbitrarily large real numbers $H$, the inequalities
\begin{equation}
\label{exponent}
\norm{ A\underline{x}-\underline{\theta}}\leq H^{-\omega}\;\;\;
\text{and }\;\;\;\abs{\underline{x}}\leq H
\end{equation}
have a solution $\underline{x}$ in $\mathbb{F}_q[T]^m$.
Let $\hatomega{A,\underline{\theta}}$ be the supremum of the real numbers $\omega$ for which,
for all sufficiently large positive real numbers $H$,
the inequalities (\ref{exponent}) have a solution $\underline{x}$ in $\mathbb{F}_q[T]^m$.
The homogeneous exponents $\omega (A)$ and $\hatomega{A}$ 
are similarly def{i}ned by taking $\thetabar = \zero$
and disallowing $\underline{Q}$ to be so.
It is then clear that for any pair $A, \thetabar$,
we have $\omega ( A, \thetabar ) \geq \hatomega{A, \thetabar} \geq 0$ and $\omega (A) \geq \hatomega{A} \geq 0$.

In language of exponent defined above,
we conclude from Theorem \ref{Th:inhom} that $\omega ( \xi, \alpha ) \geq 1$ for any irrat{i}onal $\xi$,
and that $\omega ( \xi, \alpha )=1$ for some irrat{i}onal $\xi$ and $\alpha$.
We further have
    \begin{equation}\label{E:n1}
      \omega (\thetabar) = \hatomega{\thetabar} = 1/m \textrm{ for almost every } \thetabar \in \Fqtinv^m, 
    \end{equation}
(with respect to the Haar measure) by the Borel-Cantelli lemma.

There is a lot of recent act{i}vity about understanding the Diophant{i}ne propert{i}es of group act{i}ons
on homogeneous spaces.
If we consider Theorem \ref{Th:inhom} as the study of the action of $\FqT$ on $\Fqtinv $,
Our next goal is to obtain an analogous result for the standard act{i}on of \SLtwoFqT on $\Fqtinv^2$.

\citeauthor*{GGN3}~\cite{GGN3, GGN15} have studied the generic rate of approximat{i}on by lat{t}ice orbits
for a large class of lat{t}{i}ce act{i}ons on homogeneous variet{i}es of connected almost simple, semisimple algebraic groups.
\citeauthor{LN12}~\cite{LN12} conf{i}ned their invest{i}gat{i}ons to the standard linear act{i}on
of the lat{t}ice $\mathrm{SL} ( 2, \Z )$ on the punctured plane $\R^2 \setminus \{ \zero \}$.
In a previous work~\citep{Sin17},
the second-named author extended their approach and showed similar results 
for a few lat{t}ices inside $\mathrm{SL} ( 2, \C )$ act{i}ng linearly on $\C^2 \setminus \{ \zero \}$.
The last two approaches involve making use of some cont{i}nued fract{i}on algorithm
to construct certain \emph{convergent matrices} belonging to the relevant lat{t}ice.
An alternate strategy deployed in~\cite{MW12} and \cite{Pol11} works for many more examples
as it uses ef{f}ect{i}ve equidistribut{i}on results but usually gives weaker est{i}mates.

Let $\x = ( x_1, x_2 ) \in \Fqtinv^2$ with ``slope'' $\xi := x_1 / x_2$ in $\Fqtinv$.
We consider its orbit under the standard act{i}on of \SLtwoFqT.
If the slope $\xi$ is in $\Fqt$ and $P / Q$ is its representat{i}on in `lowest terms'
for some $Q \in \FqT \setminus \{ 0 \}$,
the coordinates of any non-zero vector $\gamma\x$ shall have entries with absolute value
at least $\min \{ \abs{x_2}, 1 \}\cdot\abs{Q}^{-1}$.
Similarly, two dist{i}nct points in the \SLtwoFqT-orbit of \x\ will also be
at least $\min \{ \abs{x_2}, 1 \}\cdot\abs{Q}^{-1}$ apart and we will have a discrete orbit at hand.
We are more interested in analyzing the nature of dense orbits here.
Therefore, we assume henceforth that $\xi \in \Fqtinv \setminus \Fqt$.
Our target is to reach as close to some f{i}xed point $\y \in \Fqtinv^2$ as possible
with the help of smallest matrix size \abs{\gamma} for some $\gamma \in \SLtwoFqT$.
Just like Def{i}nit{i}on of above exponent, we have
\begin{Def}\label{D:mu}
The \emph{asymptot{i}c Diophant{i}ne exponent} $\mu ( \x, \y )$ refers to
\begin{equation}
\sup \left\{ \mu \mid \abs{\gamma \x - \y} \leq \abs{\gamma}^{-\mu} 
\textrm{ has inf.\ many solut{i}ons in } 
\gamma \in \SLtwoFqT \right\},
\end{equation}
and the \emph{uniform Diophant{i}ne exponent} \hatmu {\x}{\y} is given by
        \begin{equation}
          \sup \left\{ \hat{\mu} \mid \forall H \gg 0, \exists \gamma \in \SLtwoFqT \textrm{ such that } \abs{\gamma} \leq H, \abs{\gamma \x - \y} \leq H^{- \hat{\mu} } \right\}.
        \end{equation}
\end{Def}

Here we give the analogue of the results in \citeauthor{LN12}~\cite{LN12}.
\begin{theorem}\label{T: exponent for lattice}
      \begin{itemize}
      Let $\x = \tp{( x_1, x_2 )} \in \Fqtinv^2$ such that $\xi = x_1 / x_2$ is not in $K$.
            \item [(1)]
      We have
          \[
            \mu ( \x, \zero ) = 1\quad\text{and}\quad \hatmu{\x}{\zero} = 1 / \omega (\xi).
          \]
      \item [(2)]
      If the vector $\y = \tp{( y_1, y_2 )} \in \Fqtinv^2$ has slope $y = y_1 / y_2$ in $\Fqt$, then we have
        \[
          \mu ( \x, \y ) = \frac{\omega (\xi)}{\omega (\xi ) + 1}\quad\text{and}\quad\hatmu{\x}{\y} = \frac{1}{\omega (\xi) + 1}.
        \]
      \item [(3)]
       If the slope $y$ of the vector $\y$ is not in $K$, then we have 
        \[
          \mu ( \x, \y )\geq\frac{1}{3}\quad\text{and}\quad\hatmu{\x}{\y}
          \geq \frac{\omega(y)+1}{2(2\omega(y)+1)\omega (\xi)}\geq \frac{1}{4\omega(\xi)}.
        \]
       \end{itemize}
\end{theorem}
A generic upper bound for the asymptotic exponent $\mu(\underline{x},\underline{y})$ is given by the following theorem,
which is the analogue of Theorem 3 in \citeauthor{LN12}~\cite{LN12}.
\begin{theorem}
\label{T:generic upper bound}
Let $\underline{x}$ be a point in $\Fqtinv^2$ with irrational slope 
and let $y$ be an irrational element in $\Fqtinv$ having irrationality exponent $\omega(y)=1$.
Then we have 
$$
\mu(\underline{x},\underline{y})\leq \frac{1}{2}
$$
for almost all points $\underline{y}$ of the line $\Fqtinv ^t(\underline{y},1)$.
\end{theorem}
In the next section, we present some auxiliary results of continued fraction. In section 3, the analogue of the Minkowski's theorem is proved. The last section is devoted to giving the proof of Theorem \ref{T: exponent for lattice} and Theorem \ref{T:generic upper bound}.
%%%%%%%%%%%%%%%%%%%%%%%%%%%%%%%%%%%%%%%%%%%%%%%%%%%%%%%%%%%%%%%%%%%%%%%%%%%%%%%%%%%%%%%%%%%%%%%%%%%%%%%%%%%%%%%%%%%%%%%%%%%
%%%%%%%%%%%%%%%%%%%%%%%%%%%%%%%%%%%%%%%%%%%%%%%%%%%%%%%%%%%%%%%%%%%%%%%%%%%%%%%%%%%%%%%%%%%%%%%%%%%%%%%%%%%%%%%%%%%%%%%%%%%
\section{\protect{Cont{i}nued fract{i}ons}}\label{S:cf}
Let $\{ A_i \}_{i \in \N} \subset \FqT$ with $\deg A_i > 0$ for $i > 0$. 
The \emph{cont{i}nued fract{i}on} $\xi := [ A_0 ; A_1, \cdots ]$ is the limit of the sequence of part{i}al fract{i}ons
    \begin{equation}\label{E:cf}
      \dfrac{P_n}{Q_n} := A_0 + \cfrac{1}{A_1 + \cfrac{1}{A_2 + \cfrac{1}{\ddots + \frac{1}{A_n}}}}
    \end{equation}
as $n \rightarrow \infty$. It exists for every such sequence of \emph{part{i}al quot{i}ents} $A_i$'s and moreover
    \begin{equation}\label{E:cferror}
      \abs{Q_n \xi - P_n} = \dfrac{1}{\abs{Q_{n + 1}}}.
    \end{equation}
Conversely, given $\xi \in \Fqtinv$, such an expansion is unique. 
The rat{i}onal funct{i}ons $P_n / Q_n$ are called \emph{$n$-th order convergents} to $\xi$. 
They sat{i}sfy the recurrence
    \begin{align}\label{E:Qpair}
       P_{-2} := 0,\quad &P_{-1} := 1,\quad&P_n &= A_n P_{n - 1} + P_{n - 2} \textrm{ for } n \geq 0, \textrm{ and }\\
       Q_{-2} = 1,\quad &Q_{-1} = 0,\quad&Q_n &= A_n Q_{n - 1} + Q_{n - 2} \textrm{ for } n \geq 0.\notag
    \end{align}
Every f{i}nite expression $[ A_0; A_1, \cdots, A_N ]$ leads to an element of \Fqt. 
In the converse direct{i}on, 
it is also true that the cont{i}nued fract{i}on expansion of every rat{i}onal funct{i}on $P / Q$ terminates 
in f{i}nitely many steps because the Euclidean nature of the ring \FqT. 
By induct{i}on, $Q_n P_{n - 1} - P_n Q_{n - 1} = ( -1 )^n$ and $\abs{P_n} > \abs{P_{n - 1}}$ 
whenever all the terms are def{i}ned. 
As a consequence, $\abs{P_n} = \abs{A_n P_{n - 1}} = \prod_{i = 0}^n \abs{A_i}$. 
One can use a similar argument to establish
    \begin{equation}\label{E:absQn}
      \abs{Q_n} = \abs{A_n Q_{n - 1}} = \prod_{i = 1}^n \abs{A_i}.
    \end{equation}
For future use, 
it follows from~\eqref{E:cferror} and \eqref{E:absQn} that 
\[
  \abs{Q_n ( Q_n \xi - P_n )} = 1 / \abs{A_{n + 1}}.
\]

    \begin{notation}\label{N:sc}
      When $A_0 = 0$ as will of{t}en be the case, 
      we simply drop its ment{i}on along with the succeeding semi-colon and write 
      $[ A_1, A_2, \cdots ] := [ 0; A_1, A_2, \cdots ]$ 
      while $[A_1, A_2, \cdots, A_n ]$ shall denote the corresponding f{i}nite truncat{i}ons.
    \end{notation}
It is well known that for every $\xi \in \Fqtinv$, 
there exist inf{i}nitely many polynomials $Q$ such that $\abs{Q \xi - P} < 1 / \abs{Q}$ 
    for some $P\in \FqT$ depending upon $Q$. 
The next statement from~\citep{Sch00} tells us where to look for them.
    \begin{lem}[\citeauthor{Sch00}]\label{L:legendre}
      For any $Q \in \FqT \setminus \{ 0 \}$ such that $\abs{Q \xi - P} < 1/\abs{Q}$, 
      the rat{i}onal funct{i}on $P/Q$ is a convergent to $\xi$.
    \end{lem}
In part{i}cular, these convergents are the \emph{best approximants of second kind}~\citep{Khi64}. What this means is that for all polynomials $Q$ with $0 < \abs{Q} < \abs{Q_{n + 1}}$, we have $\abs{Q\xi - P} \geq \abs{Q_n \xi - P_n}$ for all $P \in \FqT$. For if not, let $P, Q\ ( 0 < \abs{Q} < \abs{Q_{n + 1}} )$ sat{i}sfy
    \begin{equation}\label{E:best}
      \abs{Q\xi - P} < \abs{Q_n\xi - P_n} = \frac{1}{\abs{Q_{n + 1}}} < \frac{1}{\abs{Q}}.
    \end{equation}
Then, \citeauthor{Sch00} says that such a $P / Q = P_m / Q_m$ for some $m \leq n$ which would imply
    \begin{equation}
      \abs{Q\xi - P} \geq \abs{Q_m\xi - P_m} = \frac{1}{\abs{Q_{m + 1}}} \geq \frac{1}{\abs{Q_{n + 1}}}
    \end{equation}
as $m \leq n$ and recalling~\eqref{E:absQn}. 
This however contradicts the assumpt{i}on in~\eqref{E:best}. 
Note that if all the $A_i$'s in the cont{i}nued fract{i}on expansion of $\xi$ were to be linear polynomials over $\F_q$, 
we see that the constant on the right side in \eqref{E:Dirichlet} cannot be improved for uncountably many $\xi$'s
 corresponding to the sequences $( A_i ) \in \{ T, T + 1 \}^{\N}$.\\[-0.1cm]

The quant{i}ty $\omega (\xi)$ is also known as the \emph{irrat{i}onality measure} of $\xi \in \Fqtinv \setminus \Fqt$. 
If $\omega \leq 1$, 
then the denominator sequences $\{ Q_n \}$ are strictly increasing in size and $\abs{Q_{n + 1}} \geq \abs{Q_n}^{\omega}$ 
for all $n \in \N$ trivially. 
When $1 < \omega < \omega (\xi)$, 
we have $\abs{Q \xi - P} \leq \abs{Q}^{-\omega}$ for inf{i}nitely many non-zero polynomials $Q$ and $P \in \FqT$. 
For all non-constant polynomials, $\abs{Q}^{- \omega} < \abs{Q}^{-1}$. 
Lemma~\ref{L:legendre} then says that any such fract{i}on $P / Q$ has to be a convergent to $\xi$. 
In other words, both $P$ and $Q$ are a non-zero polynomial mult{i}ple of some pair $( P_n, Q_n )$ 
and we conclude that for $1 < \omega < \omega (\xi)$ and $Q \in \FqT$ with $\abs{Q} \gg 1$,
    \begin{equation}\label{E:groQn}
      \dfrac{1}{\abs{Q_{n + 1}}} \leq \abs{R ( Q_n \xi - P_n )} = \abs{Q \xi - P} \leq \dfrac{1}{\abs{Q}^{\omega}} \leq \dfrac{1}{\abs{Q_n}^{\omega}}
    \end{equation}
where $R \in \FqT \setminus \{ 0 \}$. 
The bot{t}omline is $\abs{Q_{n + 1}} \geq \abs{Q_n}^{\omega}$ for inf{i}nitely many $n$'s. 
Lastly if $\omega > \omega (\xi)$, then in part{i}cular, 
$\abs{Q_n \xi - P_n} > \abs{Q_n}^{- \omega}$ for all $n$ large enough. 
On the other hand, we know its exact value to be $\abs{Q_{n + 1}}^{-1}$. 
We get that $\abs{Q_{n + 1}} < \abs{Q_{n}}^{\omega}$ for all $n \gg 1$.
%%%%%%%%%%%%%%%%%%%%%%%%%%%%%%%%%%%%%%%%%%%%%%%%%%%%%%%%%%%%%%%%%%%%%%%%%%%%%%%%%%%%%%%%%%%%%%%%%%%%%%%%%%%%%%%%%%%%%%%%%%%
%%%%%%%%%%%%%%%%%%%%%%%%%%%%%%%%%%%%%%%%%%%%%%%%%%%%%%%%%%%%%%%%%%%%%%%%%%%%%%%%%%%%%%%%%%%%%%%%%%%%%%%%%%%%%%%%%%%%%%%%%%%
  \section{Inhomogeneous approximat{i}on}\label{S:minkowski}
This sect{i}on is largely dedicated towards obtaining an analogous version of Minkowski's theorem 
in the field of formal power series.
Before that, we follow the proof route in the real case to establish some results related to geometry of numbers.
\begin{lem}\label{L:Cas}
      Let $\theta, \varphi, \psi\,(\neq 0)$ and $\chi$ be four formal Laurent series over $\F_q$ with
      \begin{equation}\label{E:cascond}
        \max \left\{ \abs{\theta \chi - \varphi \psi}, \abs{\psi \chi} \right\} \leq \Delta
      \end{equation}
      for some $\Delta > 0$. Then, there exists $P \in \FqT$ sat{i}sfying
      \begin{equation}\label{E:casres}
        \abs{\theta + P \psi} \abs{\varphi + P \chi} \leq q^{-1}\Delta \textrm{ as well as } 
        \abs{\theta + P \psi} \leq \abs{\psi}.
      \end{equation}
\end{lem}
    \begin{proof}
      Clearly, $\deg ( \theta + P_0 \psi ) < \deg \psi$ for some $P_0 \in \FqT$. 
      We let $\theta' := \theta + P_0 \psi$ and $\varphi' := \varphi + P_0 \chi$, 
      respect{i}vely so that $\abs{\theta' \chi - \varphi' \psi} \leq \Delta$.\\
      \textit{Case 1.} If $\abs{\varphi'} \leq \abs{\chi}$, we have imitat{i}ng \citep{Cas57}
      \begin{equation}
        16 \abs{\theta' \varphi'} \abs{( \theta' + \psi ) ( \varphi' + \chi )} 
        \leq \left( \abs{\theta'} + \abs{\theta' + \psi} \right)^2 \left( \abs{\varphi'} + \abs{\varphi' + \chi} \right)^2.
      \end{equation}
      By our construct{i}on, $\abs{\theta' + \psi} \leq \abs{\psi}$ and so is $\abs{\theta'}$. 
      We assumed $\abs{\varphi'} \leq \abs{\chi}$ which also implies that $\abs{\varphi' + \chi} \leq \abs{\chi}$. 
      The conclusion is that
      \begin{equation}\label{E:cas1}
        \abs{\theta' \varphi'} \abs{( \theta' + \psi ) ( \varphi' + \chi )} < \abs{\psi}^2 \abs{\chi}^2 \leq \Delta^2.
      \end{equation}
      \textit{Case 2.} Else if $\abs{\varphi'} > \abs{\chi}$, 
      it follows that $\abs{\theta' \chi } < \abs{\varphi' \psi}$ 
      which in turn means that $\abs{\theta' \chi - \varphi' \psi} = \abs{\varphi' \psi}$. 
      The A.~M.\,--\,G.~M.\ inequality dictates
      \begin{align}
        2 \left( \abs{\theta' \varphi'} \abs{( \theta' + \psi ) ( \varphi' + \chi )} \right)^{\frac{1}{2}} &\leq \abs{\varphi' ( \theta' + \psi )} + \abs{\theta' ( \varphi' + \chi )}\label{E:cas2}\\
        &\leq \abs{\varphi' \psi} + \abs{\theta' \varphi'} < 2\Delta.\notag
      \end{align}
      The equat{i}ons~\eqref{E:cas1} and \eqref{E:cas2} together give us that $\min \left\{ \abs{\theta' \varphi'}, 
      \abs{( \theta' + \psi ) ( \varphi' + \chi )} \right\}$ is less than $\Delta$. 
      We remind the reader that $\max \left\{ \abs{\theta'}, \abs{\theta' + \psi} \right\} \leq \abs{\psi}$. 
      Otherwise said, 
      one of the subst{i}tut{i}ons $P = P_0$ or $P = P_0 + 1$ in~\eqref{E:casres} proves our claim.
    \end{proof}
   In the proof, we need the following version of Minkowski's linear forms theorem.
    \begin{theorem}[\cite{Spr69, Zhe17}]\label{C:lf}
Let $A=(a_{i,j})_{n\times n}$ be an $n\times n$ matrix with entries in \Fqtinv\ 
and $\underline{r}=(r_1,r_2,\cdots,r_n)$ be an $n-$tuple of integers. 
If
      \[
        0 < \abs{\det(A)} < q^{-(r_1+r_2+\cdots+r_n)},
      \]
then there is a non-zero integral point $\underline{u}$ such that $|L_i(\underline{u})|<q^{-r_i} $ for all $1\le i\le n$, 
where the linear forms $L_i(\underline{x})$ are determined by the rows of the matrix $A$.
    \end{theorem}
    \begin{theorem}\label{Th:minkowski}
      Let $L_j (\P) := \lambda_j P_1 + \kappa_j P_2$ for $j = 1, 2$ and $\Delta = \abs{\lambda_1 \kappa_2 - \lambda_2 \kappa_1} > 0$.
      \begin{enumerate}
        \item For all $\rho_1, \rho_2 \in \Fqtinv, \exists \underline{Q} \in \FqT^2$ such that
        \begin{equation}\label{E:inhom}
          \abs{L_1 (\underline{Q}) + \rho_1}\abs{L_2 (\underline{Q}) + \rho_2} \leq q^{-2}\Delta.
        \end{equation}
        \item If moreover $\kappa_1 \lambda_1^{-1} \notin \Fqt$ and $k \in \N$, 
        there exists $\underline{Q}$ which also makes
        \begin{equation}\label{E:irr}
          \deg \left( L_1 (\underline{Q}) + \rho_1 \right) < -k.
        \end{equation}
      \end{enumerate}
    \end{theorem}
    \begin{proof}
      Assume $\kappa_1 \lambda_1^{-1} \notin \Fqt$ to begin with. 
      Theorem~\ref{C:lf} tells us that there is a non-zero $\P = ( P_1, P_2 ) \in \FqT^2$ for which
      \begin{equation}\label{E:min}
        \abs{L_1 (\P)} < q^{-k},\textrm{ and } \abs{L_2 (\P)} \leq q^k \Delta.
      \end{equation}
      We can assume that $\gcd ( P_1, P_2 ) = 1$ without any loss of generality. 
      Therefore, pick any $( R_1, R_2 ) \in \FqT^2$ for which
      \begin{equation}\label{E:det}
        \det \begin{pmatrix}
               P_1 & R_1\\
               P_2 & R_2
             \end{pmatrix} = 1
      \end{equation}
      and we have the transformed system
      \begin{equation}\label{E:ts}
        \begin{pmatrix}
          \,L_1 (\underline{Q})\,\\
          \,L_2 (\underline{Q})\,
        \end{pmatrix} =
        \begin{pmatrix}
          \lambda_1 & \kappa_1\\
          \lambda_2 & \kappa_2
        \end{pmatrix}
        \begin{pmatrix}
          Q_1\\
          Q_2
        \end{pmatrix} =
        \begin{pmatrix}
          \lambda'_1 & \kappa'_1\\
          \lambda'_2 & \kappa'_2
        \end{pmatrix}
        \begin{pmatrix}
          Q'_1\\
          Q'_2
        \end{pmatrix}
      \end{equation}
      for all $\underline{Q} = ( Q_1, Q_2 ) \in \FqT^2$, where
      \begin{equation}\label{E:Qprime}
        \begin{pmatrix}
          Q'_1\\
          Q'_2
        \end{pmatrix} :=
        \begin{pmatrix}
          P_1 & R_1\\
          P_2 & R_2
        \end{pmatrix}^{-1}
        \begin{pmatrix}
          Q_1\\
          Q_2
        \end{pmatrix}.
      \end{equation}
      It is plain that $( Q'_1, Q'_2 ) \in \FqT^2$ if and only if $( Q_1, Q_2 )$ does too and also that $\det
      \begin{pmatrix}
        \lambda_1 & \kappa_1\\
        \lambda_2 & \kappa_2
      \end{pmatrix} = \det
      \begin{pmatrix}
        \lambda'_1 & \kappa'_1\\
        \lambda'_2 & \kappa'_2
      \end{pmatrix}$ owing to~\eqref{E:det}. 
      Further, $\abs{\lambda'_1} = \abs{L_1 (\underline{P})} < q^{-k}$ 
      and similarly $\abs{\lambda'_2} \leq q^k \Delta$ from~\eqref{E:min}, \eqref{E:ts} and \eqref{E:Qprime}. 
      The former is non-zero since $\kappa_1 \lambda_1^{-1}$ was taken to be an irrat{i}onal funct{i}on in $\invT$.\\[-0.2cm]

      Let $Q'_2 \in \FqT$ be such that 
      $\abs{\rho_1 \lambda'_2 - \rho_2 \lambda'_1 - Q'_2 (\lambda_1 \kappa_2 - \lambda_2 \kappa_1)} \leq q^{-1}\Delta$. 
      Also, we subst{i}tute in Lemma~\ref{L:Cas}
      \begin{align}
        \theta = \kappa'_1 Q'_2 + \rho_1,\quad &\varphi = \kappa'_2 Q'_2 + \rho_2\label{E:subst}\\
        \psi = \lambda'_1,\quad &\textrm{and } \quad \chi = \lambda'_2\notag
      \end{align}
      to get that $\abs{\theta \chi - \psi \varphi} 
      = \abs{\rho_1 \lambda'_2 - \rho_2 \lambda'_1 - Q'_2 (\lambda_1 \kappa_2 - \lambda_2 \kappa_1)} \leq q^{-1}\Delta$ 
      as well as $\abs{\psi \chi} = \abs{\lambda'_1 \lambda'_2} \leq q^{-1}\Delta$. We already argued $\psi \neq 0$. 
      Thus, there exists some $Q'_1 \in \FqT$ (and the associated pair $( Q_1, Q_2 )$ determined by~\eqref{E:Qprime}) 
      for which
      \begin{align}
        \abs{L_1 (\underline{Q}) + \rho_1}\abs{L_2 (\underline{Q}) + \rho_2 } &= \abs{\lambda'_1 Q'_1 + \kappa'_1 Q'_2 + \rho_1}\abs{\lambda'_2 Q'_1 + \kappa'_2 Q'_2 + \rho_2}\notag\\
        &= \abs{\theta + \psi Q'_1}\abs{\varphi + \chi Q'_1} \leq q^{-2}\Delta \textrm{ and}\\
        \abs{L_1 (\underline{Q}) + \rho_1} &= \abs{\theta + \psi Q'_1} \leq \abs{\psi} < q^{-k}.\notag
      \end{align}
      If it happens that $\lambda_1 P_1 + \kappa_1 P_2 = 0$ (when $\kappa_1 \lambda_1^{-1} \in \Fqt$), 
      we can be sure that $\lambda_2 P_1 + \kappa_2 P_2 \neq 0$ as $\Delta \neq 0$. 
      Hence, we only need to exchange the roles of $L_1$ and $L_2$ amongst themselves 
      and the conclusion in~\eqref{E:inhom} remains valid.
    \end{proof}
\noindent The constant on the right side in~\eqref{E:inhom} is the smallest possible in general as follows easily 
from the observat{i}on that $\abs{P_1 + \invT}\abs{P_2 + \invT} \geq q^{-2}$ for all $P_1, P_2 \in \FqT$ (here $\Delta = 1$).
We are now in a posit{i}on to prove the promised version of Minkowski's result on inhomogeneous Diophant{i}ne approximat{i}on.

\begin{proof}[Proof of Theorem \ref{Th:inhom}]
Let $L_1 (\underline{Q}) = \xi Q_1 + Q_2,\ L_2 (\underline{Q}) = Q_1,\ \rho_1 = \alpha$ 
and $\rho_2 = 0$ in Theorem~\ref{Th:minkowski}. 
We have $\Delta = 1$ and as $\xi \notin \Fqt$ and $\alpha \neq Q\xi + P$ for any $P, Q \in \FqT$, 
the solut{i}on set $\big\{ \big( Q^{(k)}_1, Q^{(k)}_2 \big) \big\}$ to~\eqref{E:inhom} 
corresponding to dif{f}erent $k$ in~\eqref{E:irr} is inf{i}nite.

Now we turn to prove the second part.
Let $\xi = [ A_1, A_2, \cdots ]$ such that $\deg A_i > 0$ for all $i \geq 1$ and $T \mid A_i$ in \FqT\ for all $i$. 
It is clear that $\abs{\xi} = 1 / \abs{A_1}$. Now, suppose
      \begin{equation}\label{E:rea}
        q^2 \abs{Q ( Q\xi - \alpha - P )} = \abs{TQ} \abs{( TQ + 1 )\xi - ( TP + 1 )} < 1
      \end{equation}
for some non-zero polynomial $Q$ and $P \in \FqT$. 
Then, \abs{TQ + 1}\ has to be at least $\abs{A_1} = \abs{Q_1} \geq \abs{Q_1 Q_0}^{1/2}$ or else, 
$\abs{( TQ + 1 )\xi} < 1$ rendering~\eqref{E:rea} untrue. 
We thereby have a unique $n \geq 1$ such that
      \begin{equation}\label{E:sandwich}
        \abs{Q_n Q_{n -1}}^{1/2} \leq \abs{TQ} = \abs{TQ + 1} < \abs{Q_{n + 1} Q_n}^{1/2}.
      \end{equation}
 As $\abs{P_{n - 1}Q_n - P_n Q_{n - 1}} = 1$, we have a unique pair $( U, V ) \in \FqT^2$ sat{i}sfying
      \begin{equation}\label{E:UV}
        \begin{pmatrix}
          P_{n - 1} & P_n\\
          Q_{n - 1} & Q_n
        \end{pmatrix}
        \begin{pmatrix}
          U\\ V
        \end{pmatrix} =
        \begin{pmatrix}
          TP + 1\\ TQ + 1
        \end{pmatrix}
      \end{equation}
      which obey
      \begin{align}\label{E:Ubd}
        \abs{U} &= \abs{( TP + 1 )Q_n - P_n ( TQ + 1 )}\\
                &= \abs{( TQ + 1 ) ( Q_n \xi - P_n ) - Q_n \big( ( TQ + 1 )\xi - ( TP + 1 ) \big)} < \abs{A_n}^{1/2}\notag
      \end{align}
      and using~\eqref{E:sandwich} again,
      \begin{equation}\label{E:Vbd}
        \abs{V} = \abs{( TP + 1 )Q_{n - 1} - P_{n - 1} ( TQ + 1 )} < \abs{A_{n + 1}}^{1/2}.
      \end{equation}
 Our assumpt{i}on about the part{i}al quot{i}ents $A_i$'s of $\xi$ gives 
      that exactly one of the polynomials $P_n$ and $Q_n$ is divisible by $T$ in \FqT\ for each $n$. 
 This implies neither $U$ nor $V$ can be zero in~\eqref{E:UV} and in turn that $\abs{UV} \geq 1$. 
 By~\eqref{E:cferror}, \eqref{E:sandwich} and \eqref{E:UV}, we have
      \begin{align}
        \abs{TQ}&\abs{( TQ + 1 )\xi - ( TP + 1 )}\notag\\
          &= \abs{\frac{( TQ + 1 )}{Q_n}\big( \frac{(TQ + 1 )\xi - ( TP + 1 )}{Q_{n - 1}\xi - P_{n - 1}} \big)}\notag\\
        &= \abs{U\frac{Q_{n - 1}}{Q_n} + V}\abs{U + V\frac{Q_n \xi - P_n}{Q_{n - 1}\xi - P_{n - 1}}}\\
        &= \abs{UV \big( 1 + \frac{Q_{n - 1} ( Q_n \xi - P_n )}{Q_n ( Q_{n - 1}\xi - P_{n - 1} )} \big) + U^2 \frac{Q_{n - 1}}{Q_n} + V^2 \frac{Q_n \xi - P_n}{Q_{n - 1}\xi - P_{n - 1}}}.\notag
      \end{align}
The f{i}rst term in the last expression has absolute value at least $1$ as 
      we argued $\abs{Q_{n + 1}} > \abs{Q_n} > \abs{Q_{n -1}}$ before. 
The summands involving $U^2$ and $V^2$ have strictly smaller norm because of~\eqref{E:cferror}, 
\eqref{E:absQn}, \eqref{E:Ubd} and \eqref{E:Vbd}. 
We thus have a contradict{i}on to the hypothesis in~\eqref{E:rea}.
    \end{proof}
    The following statement, whose proof is omi{t}ted, also implies that our bound in Minkowski's theorem is the best possible.
    \begin{prop}
      Let $\xi = [ A_1, A_2, \cdots ] \in T^{-1}\FqTinv \setminus \Fqt$ be such that there is a non-constant, irreducible polynomial $R$ which divides $A_i$ in \FqT\ for all $i > 0$. Then, there exists an $\alpha \in \Fqtinv$ such that
      \[
        \norm{Q\xi - \alpha} \geq \abs{R}^{-2}\abs{Q}^{-1}
      \]
      for all $Q \in \FqT \setminus \{ 0 \}$.
    \end{prop}
    In this theorem, it implies that $\omega ( \xi, \alpha ) = 1$ for such a $( \xi, \alpha )$. Actually,
this is also the value of the asymptot{i}c exponent for any $\xi \in \Fqtinv \setminus \Fqt$ and almost all $\alpha$.
Let us f{i}rst observe:
    \begin{prop}
      For any irrational element $\xi$ in \Fqtinv, we have $\hatomega{\xi}=1$.
    \end{prop}
    \begin{proof}
We denote the (inf{i}nite) sequence of convergents to $\xi$ by $(P_k/Q_k)_{k\geq 1}$ as before.
For $k$ suf{f}iciently large, let $Q$ be any non-zero polynomial with $\abs{Q} \leq q^{-1}\abs{Q_{k}} =: H_k$.
Because the convergents are well-known to be the best approximants of second kind,
it follows that
      \begin{equation}
        \|Q\xi\|\geq \|Q_{k-1}\xi\|=\frac{1}{|Q_k|}= \frac{1}{qH_k}.
      \end{equation}
This shows that $\hatomega{\xi}$ can be at most equal to $1$.
The fact that it is equal to $1$ is then obvious from Dirichlet's Theorem.
    \end{proof}
    Af{t}er this, we invoke Theorem~1.2 of \cite{BZ19} which states
    \begin{theorem}
      Let $A \in M_{m \times n} ( \Fqtinv )$ and $\underline{\theta} \in \Fqtinv^m$. Then,
      \[
        \omega ( A, \underline{\theta} ) \geq \frac{1}{\hatomega{\tp{A}}}\quad
        \text{and}\quad\hatomega{A, \underline{\theta}} \geq \frac{1}{\omega ( \tp{A )}}
      \]
      with both inequalit{i}es actually being equalit{i}es for (Haar-) almost all $\underline{\theta} \in \Fqtinv^m$.
    \end{theorem}

Our next endeavour is to prove that there is no uniform posit{i}ve lower bound for the funct{i}on $\hatomega{\xi, \alpha}$. 
The proposit{i}on below is in the spirit of Theorem~III of~\citep[Chap.~3]{Cas57} 
and our proof strategy borrows heavily from theirs.
    \begin{prop}\label{P:cas}
      Let $\Psi : \N \rightarrow \R_{> 0}$ be an approximat{i}ng funct{i}on with $\Psi (x) \rightarrow 0$ 
      as $x \rightarrow \infty$. Then, there exists a pair $( \xi, \alpha ) \in \Fqtinv^2$ such that the system
      \[
        \norm{Q\xi - \alpha} \leq \Psi (H),\quad \abs{Q} \leq H
      \]
      does not have a solut{i}on for inf{i}nitely many $H$.
    \end{prop}
    \begin{proof}
      We f{i}x $\alpha = \invT$ and our desired element $\xi$ shall be the limit of a sequence of rat{i}onal funct{i}ons 
      $R_n / S_n, n \in \N$, 
      where $S_n \in \FqT \setminus T\,\FqT$ for all $n$. 
      We note that $\norm{Q \frac{R_n}{S_n} - \alpha} \geq ( q\abs{S_n} )^{-1}$ for all elements of this sequence 
      and any $Q \in \FqT$. 
      Let $R_0 / S_0 = 0$ and $R_1 / S_1 = ( T + 1 )^{-1}$. 
      In parallel, we construct a sequence $\{ H_n \}_{n \geq 2} \subset \N$ as follows:\\[-0.2cm]

      Assuming that $R_n, S_n, H_n$ have been def{i}ned for all $n \leq N$, let $H_{N + 1}$ be the smallest for which
      \begin{align}
        \Psi ( H_{N + 1} ) &< \frac{1}{q^2 \abs{S_N}}\quad ( N \geq 1 ), \textrm{ and}\label{E:nextH}\\
	    H_{N + 1} &> H_N\quad ( N > 1 ).\notag
      \end{align}
      As $\Psi (x) \rightarrow 0$ at inf{i}nity, 
      such a number can be found. 
      Now, choose any $S_{N + 1}$ with a non-zero constant term and $\abs{S_{N + 1}} \geq q H_{N + 1}\abs{S_N}$ 
      and let $R_{N + 1}$ equal the integral part $[ S_{N + 1} R_N S_N^{-1} ]$. 
      Then,
      \begin{equation}\label{E:error}
        \abs{\dfrac{R_{N + 1}}{S_{N + 1}} - \dfrac{R_{N}}{S_{N}}} \leq \dfrac{1}{q\abs{S_{N + 1}}} 
        \leq \dfrac{1}{q^2 \abs{S_N}H_{N + 1}}
      \end{equation}
      and the limit 
      $\xi := \lim_{n \rightarrow \infty} R_n / S_n = \lim_{n \rightarrow \infty} \sum_{i = 1}^n \left( R_i / S_i - R_{i - 1} / S_{i - 1} \right)$ exists. Moreover,
      \begin{equation}\label{E:xierror}
        \abs{\xi - \dfrac{R_n}{S_n}} = \lim_{m \rightarrow \infty} \abs{\dfrac{R_{n + m}}{S_{n + m}} - \dfrac{R_n}{S_n}} \leq \dfrac{1}{q^2 \abs{S_n} H_{n + 1}}
      \end{equation}
      because of the ultrametric inequality. 
      If now $Q \in \FqT$ with $\abs{Q} \leq H_{n + 1}$ and $P \in \FqT$ be any polynomial,
      \begin{align}\label{E:final}
        \abs{Q \xi - \alpha - P} &\geq \abs{Q \dfrac{R_n}{S_n} - \alpha - P} - \abs{Q}\abs{\xi - \dfrac{R_n}{S_n}}\notag\\
	                         &\geq \norm{Q \dfrac{R_n}{S_n} - \alpha} - \dfrac{1}{q^2\abs{S_n}}\\
				 &\geq \dfrac{1}{q\abs{S_n}} - \dfrac{1}{q^2\abs{S_n}} \geq \dfrac{1}{q^2\abs{S_n}} > \Psi ( H_{n + 1} )\notag
      \end{align}
      using~\eqref{E:nextH} and \eqref{E:xierror}. 
      Since this is true for any $P$, 
      the sequence $\{ H_n \}$ const{i}tutes the required set of inf{i}nitely many insoluble cases.
    \end{proof}
    The following result is also implied by Theorem 2.3 of \cite{BZ19}. Here, we give another simple proof.
    \begin{cor}
      There exists $( \xi, \alpha )$ such that $\hatomega{\xi, \alpha} = 0$.
    \end{cor}
\begin{proof}
 Let $\Psi (n) := n^{- 1 / \log \log n}$. 
 It clearly goes to zero as $n$ tends to $\infty$. 
 In addi{t}ion for any $\eps > 0,\ \Psi (n) > n^{- \eps}$ eventually and hence, 
 the pair $( \xi, \alpha )$ corresponding to $\Psi$ given to us by Prop.~\ref{P:cas} has $\hatomega{\xi, \alpha} \leq \eps$. 
 Our choice of \eps\ was arbitrary.
 \end{proof}
  We will like to end this sect{i}on with a brief discussion on the issue of monicity. This is to say that our concern is to f{i}nd solut{i}ons to the inhomogeneous inequality
  \begin{equation}
    \abs{Q\xi - \alpha - P} < \eps
  \end{equation}
  when the polynomial $Q$ is restricted to be monic and $\xi \notin \Fqt$. The argument given below follows closely that of Kronecker for real numbers.
  \begin{proof}[Proof of Proposi{t}ion~\ref{P:monic}]
    From the funct{i}on f{i}eld Dirichlet's theorem, one knows that there are inf{i}nitely many polynomial pairs $( R, S )$ for which
    \begin{equation}
      S \xi - R = \xi_1 / S
    \end{equation}
    with $\abs{\xi_1} \leq 1/q$ and $\gcd ( R, S ) = 1$. We may assume here that $S$ is monic because homogeneity. The required polynomials for the inhomogeneous problem may be obtained by perturbing each of these $S$ where the amount of perturbat{i}on is determined by $\alpha$. Consider the polynomial part $[ S\alpha ]$. As $R$ and $S$ are co-prime in \FqT, we will be able to f{i}nd polynomials $P_1$ and $P_2$ such that $\abs{P_1} < \abs{S}$ and
    \begin{equation}
      [ S\alpha ] = R  P_1 + S P_2.
    \end{equation}
    Then, we have
    \begin{equation}
      S ( P_1 \xi + P_2 ) = RP_1 + SP_2 + \xi_1 / S = S\alpha + \alpha_1 + \xi_1 / S
    \end{equation}
    where again $\abs{\alpha_1} \leq 1/q$. On rearranging,
    \begin{equation}
    \abs{P_1 \xi - \alpha + P_2} \leq \frac{1}{q\abs{S}}.
    \end{equation}
    This bound does not change if we add the term $S\xi - R$ to the lef{t}-side expression. Furthermore, the polynomial $S + P_1$ is monic as well as $\abs{S + P_1} = \abs{S}$.
  \end{proof}
  The above proposi{t}ion implies that the set $\{ \braket{Q\xi} \mid Q \in \FqT,\text{ monic }\}$ is also dense in $\invT\FqTinv$. We can moreover derive a corollary from this result. Given $\eps > 0$, choose any monic $\eta \in \Fqtinv \setminus \Fqt ( \xi, \alpha )$ with $\abs{\eta} < \eps$ and replace $\alpha$ by $\alpha + \eta$ in Proposi{t}ion~\ref{P:monic}. The inf{i}nitely many monic solut{i}ons in $Q$ whose norm is more than $1/\eps$ will all help us to realize the constrained inequali{t}y
  \begin{equation}
    \abs{Q\xi - \alpha - P} < \eps
  \end{equation}
  under the demand that $Q$ and $Q\xi - \alpha - P$ be monic.
%%%%%%%%%%%%%%%%%%%%%%%%%%%%%%%%%%%%%%%%%%%%%%%%%%%%%%%%%%%%%%%%%%%%%%%%%%%%%%%%%%%%%%%%%%%%%%%%%%%%%%%%%%%%%%%%%%%%%%%%%%%
%%%%%%%%%%%%%%%%%%%%%%%%%%%%%%%%%%%%%%%%%%%%%%%%%%%%%%%%%%%%%%%%%%%%%%%%%%%%%%%%%%%%%%%%%%%%%%%%%%%%%%%%%%%%%%%%%%%%%%%%%%%
     \section{Exponents for \SLtwoFqT\ act{i}on}\label{S:SL2}
     Without any loss of generality, we assume that the start{i}ng point \x\ equals $\tp{( \xi, 1 )}$. 
If needed, we can also use the matrix
        \begin{equation}\label{E:J}
          J := \begin{pmatrix}
                 0 & -1\\
                 1 & 0
               \end{pmatrix}
        \end{equation}
to have $\abs{\xi} \leq 1$. 
This will mean that $\abs{P_k} \leq \abs{Q_k}$ for all $k \geq 0$, where $P_k / Q_k$ 
is the $k$-th convergent to $\xi$. 
Being a (signed) permutat{i}on matrix, 
$J$ has the desirable property that $\abs{J\x} = \abs{\x}$ for all $\x$ in $\Fqtinv^2$ 
as well as $\abs{J\gamma} = \abs{\gamma J} = \abs{\gamma}$ for all $2 \times 2$ matrices $\gamma$.
        \begin{lem}\label{L:lb}
          Let $\gamma \in \SLtwoFqT$ be such that $\abs{\gamma} < \abs{Q_{k + 1}}$ 
          where $P_k / Q_k $ are convergents to $\xi$. Then,
          \[
            \abs{\gamma \x} \geq \frac{1}{\abs{Q_k}}.
          \]
        \end{lem}
\noindent The argument here is same as the one used in \cite{LN12, Sin17} except that 
we get tighter bounds owing to the ultrametric inequality. 
Now if $k$ is chosen so as to have $\abs{Q_k} \leq \abs{\gamma} < \abs{Q_{k + 1}}$, 
we immediately get $\mu ( \x, \zero ) \leq 1$. At the same time,
        \begin{equation}
          M_k \x = \begin{pmatrix}
                      Q_k \xi - P_k\\
                     (-1)^{k - 1} ( Q_{k - 1} \xi - P_{k - 1} )
                   \end{pmatrix}
        \end{equation}
        for
        \begin{equation}\label{E:conv}
          M_k := \begin{pmatrix}
                    Q_k & - P_k\\
                   (-1)^{k - 1} Q_{k - 1}  & (-1)^{k} P_{k - 1}
                 \end{pmatrix}
        \end{equation}
        and therefore, $\abs{M_k \x} = \abs{\eps_{k - 1}} = 1 / \abs{Q_k}$ for all $k \in \N$ 
        where $\eps_n := Q_n \xi - P_n$ is the scaled error (with sign) for approximat{i}on by the $n$-th convergent of $\xi$. By our assumpt{i}on that $\abs{\xi} \leq 1$,
         we get $\abs{P_{k - 1}} \leq \abs{Q_{k - 1}},\ \abs{P_k} \leq \abs{Q_k}$ and $\abs{Q_{k - 1}} < \abs{Q_k} = \abs{M_k}$. 
         Thereby, one has $\abs{M_k \x} = 1 / \abs{M_k}$. As $k \rightarrow \infty$, 
         we get inf{i}nitely many such matrices in \SLtwoFqT. The conclusion is
        \begin{equation}
          \mu ( \x, \zero ) = 1.
        \end{equation}
        This also gives us an upper bound for the uniform exponent $\hatmu{\x}{\zero}$. 
        We can actually improve upon this. 
        If $\omega < \omega (\xi)$, 
        it follows from the definition of $\omega(\xi)$ 
         that $\abs{Q_{k + 1}} \geq \abs{Q_k}^{\omega}$ for inf{i}nitely many $k$. 
        Consider $H = \abs{Q_{k + 1}} / q$ for such a $k$ so that for all matrices $\gamma \in \SLtwoFqT$ 
        with $\abs{\gamma} \leq H$,
        \begin{equation}
          \abs{\gamma\x} \geq \frac{1}{\abs{Q_k}} \geq \frac{1}{\abs{Q_{k + 1}}^{1 / \omega}} = \frac{1}{( qH )^{1 / \omega}}
        \end{equation}
        by Lemma~\ref{L:lb}. 
        Said differently, $\hatmu{\x}{\zero}$ can be at most $1 / \omega$. 
        This is strengthened to $1 / \omega (\xi)$ by let{t}ing $\omega \rightarrow \omega (\xi)$ from below. 
        For the lower bound, let $T \gg 1$ be so that there exists some $k \in \N$ for 
        which $\abs{Q_k} = \abs{M_k} \leq T < \abs{M_{k + 1}} = \abs{Q_{k + 1}}$. 
        We then have
        \begin{equation}
          \abs{M_k\x} = \frac{1}{\abs{Q_k}} < \frac{1}{\abs{Q_{k + 1}}^{1 / \omega}} < \frac{1}{T^{1 / \omega}}
        \end{equation}
        for all $k \gg 1$ and $\omega > \omega (\xi)$ from the discussion in \S\,\ref{S:cf}. 
        Let $\omega$ approach $\omega (\xi)$ from the right and we can write down
        \begin{prop}
          Let $\x = \tp{( x_1, x_2 )} \in \Fqtinv^2$ such that $\xi = x_1 / x_2$ is not in $K$. 
                    Then,
          \[
            \mu ( \x, \zero ) = 1\quad\text{and}\quad \hatmu{\x}{\zero} = 1 / \omega (\xi).
          \]
        \end{prop}
\noindent From~\eqref{E:n1}, 
we see that the two exponents are equal for almost all $\x$ with respect to the Haar measure. 
Our next goal is to bound the size of an \SLtwoFqT\ matrix in terms of convergent 
and upper-triangular matrices. In the sequel, for $a$ in $\FqT$, we set 
$$
U(a) = \begin{pmatrix} 1 & a \\ 0 & 1 \end{pmatrix}. 
$$

        \begin{lem}[cf.~\cite{LN12}]\label{L:gamma}
          Let $k \in \N$  and
          \begin{equation}\label{E:N}
            N = \begin{pmatrix} t & t'\\ s & s' \end{pmatrix}
          \end{equation}
be any matrix in \SLtwoFqT. Then, the product $\gamma = N U(a) M_k$ satisfies 
          \begin{equation}
            \abs{\big( a Q_{k - 1} + ( -1)^{k - 1} Q_k \big) s } - \abs{s'Q_{k - 1}} 
            \leq \abs{\gamma} \leq \abs{N}\cdot\max \big\{ \abs{Q_k}, \abs{aQ_{k - 1}} \big\}.
          \end{equation}
        \end{lem}
\begin{proof}
As we have seen before in~\cite{LN12} and \cite{Sin17},
          \begin{align}
            \gamma &= \begin{pmatrix} t & t'\\ s & s' \end{pmatrix}
                      \begin{pmatrix} 1 & a\\ 0 & 1\end{pmatrix}
                      \begin{pmatrix}
                        Q_k & - P_k\\
                        (-1)^{k - 1} Q_{k - 1} & (-1)^{k} P_{k - 1}
                      \end{pmatrix}\\
                   &= \begin{pmatrix}
                        tQ_k + (-1)^{k - 1}Q_{k - 1} ( ta + t' ) & -t P_k + (-1)^{k}P_{k - 1} ( ta + t' )\\
                        sQ_k + (-1)^{k - 1}Q_{k - 1} ( sa + s' ) & -sP_k + (-1)^{k}P_{k - 1} ( sa + s' )
                      \end{pmatrix}.\notag
          \end{align}
The lower bound is given to us by the bot{t}om left entry of this last matrix. Since \abs{\xi} was assumed to be at most $1$ and irrat{i}onal, 
$\abs{P_n} \leq \abs{Q_n} < \abs{Q_{n + 1}}$ for all $n \in \N$. 
This leads to the upper bound.
\end{proof}
Being done with that, we now want an upper bound on the size of the vector $\gamma\x$ itself. 
The statement below is the funct{i}on f{i}eld analogue of~\citep[Lemma~3]{LN12}. 
        \begin{lem}\label{L:diff}
          Let $a, k, N$ and $\gamma = NU(a)M_k =: \begin{pmatrix} V_1 & U_1\\ V_2 & U_2\end{pmatrix}$ be 
          as in Lemma~\ref{L:gamma}. 
          For $y \in \Fqtinv$, def{i}ning $\delta = sy - t$ and $\delta' = s'y - t'$ gives us
          \[
            \abs{V_1\xi + U_1 - y ( V_2\xi + U_2 )} \leq \max \left\{ \abs{\frac{\delta}{Q_{k + 1}}}, \abs{\frac{\delta a + \delta'}{Q_{k}}} \right\}.
          \]
        \end{lem}
        \begin{proof}
          We have
          \begin{align}\label{E:diff}
            y ( V_2\xi + U_2 ) - ( V_1\xi + U_1 ) &= \begin{pmatrix} -1 & y\end{pmatrix}\gamma
                                                     \begin{pmatrix} \xi\\ 1\end{pmatrix} =
                                                     \begin{pmatrix} -1 & y\end{pmatrix}
                                                     \begin{pmatrix} t & t'\\ s & s' \end{pmatrix} U(a) M_k
                                                     \begin{pmatrix} \xi\\ 1\end{pmatrix}\notag\\
                                                  &= \begin{pmatrix} \delta & \delta'\end{pmatrix} U(a)
                                                     \begin{pmatrix} \eps_k\\ ( -1 )^{k - 1} \eps_{k - 1} \end{pmatrix}\\
                                                  &= \delta  \eps_k + (-1)^{k - 1}  (\delta a + \delta' ) \eps_{k - 1}.\notag
          \end{align}
          In order to f{i}nish the argument, 
          one only needs to use $\abs{\eps_n} = 1 / \abs{Q_{n + 1}}$ for all $n \in \N$ and the ultrametric property.
        \end{proof}
        Consider any vector $\y = \tp{( y_1, y_2 )} \in \Fqtinv^2$ 
        and let $\tp{( \Lambda_1, \Lambda_2 )}$ denote the di{f}ference $\gamma \x - \y$, i.~e.,
        \begin{equation}\label{E:lambda}
          \Lambda_i = x_2 ( V_i \xi + U_i ) - y_i \text{ for } i \in \{ 1, 2 \},
        \end{equation}
        where we take $\x = \tp{( x_1, x_2 )}$ such that $\abs{\x} = \abs{x_2} > 0$ without loss of generality. 
        In part{i}cular, 
        when we choose $y$ to be the slope $y_1 / y_2$ of our target \y\ (again assume $\abs{y} \leq 1$ using the matrix $J$), 
        Lemma~\ref{L:diff} tells us that
        \begin{align}\label{E:first}
          \abs{\Lambda_1 - y \Lambda_2} &= \abs{x_2 \big( ( V_1\xi + U_1 ) - y ( V_2\xi + U_2 ) \big)}\notag\\
                                        &\leq \abs{x_2}\cdot\max \left\{ \abs{\frac{\delta}{Q_{k + 1}}}, \abs{\frac{\delta a + \delta'}{Q_{k}}} \right\}.
        \end{align}
        The idea is simple. 
        To get a bound on the size of \tp{( \Lambda_1, \Lambda_2 )}, 
        we separately bound each of the component $\Lambda_2$ and the quant{i}ty \abs{\Lambda_1 - y\Lambda_2}. 
        From~\eqref{E:diff}, we deduce that 
        \begin{align}\label{E:Lambda2}
          \Lambda_2 &= x_2 ( V_2 \xi + U_2 ) - y_2 = x_2 \big( s \eps_k + (-1)^{k - 1} ( sa + s' ) \eps_{k - 1} \big) - y_2\notag\\
          &= (-1)^{k - 1}  x_2 s \eps_{k - 1} ( a - \rho ),
        \end{align}
        where
        \begin{equation}\label{E:rho}
          \rho := \frac{(-1)^{k - 1}y_2}{x_2s\eps_{k - 1}} - \frac{(-1)^{k - 1}\eps_k}{\eps_{k - 1}} - \frac{s'}{s}.
        \end{equation}
        The element $\rho \in \Fqtinv$ is the one which decides the value of $a$ for us, 
        namely  we take $a = [ \rho ]$ so that $\abs{a - \rho} \leq 1/q$ and $\abs{a} \leq \abs{\rho}$. 
        Such a choice means $\abs{\Lambda_2} < \abs{ x_2 s / Q_k }$.

\subsection{\protect{Target points with slope in \Fqt}} Let \y\ have slope $y = y_1 / y_2 = A / B \in \Fqt$ such that $A, B \in \FqT,\ \gcd \{ A, B \} \in \F_q^*$ and $\abs{A/B} \leq 1$. Now, assign
      \begin{equation}
        N = \begin{pmatrix} A & \widetilde{A}\\ B & \widetilde{B} \end{pmatrix}
      \end{equation}
      where $\widetilde{A}$ and $\widetilde{B}$ are (upto mult{i}plicat{i}on by $\pm 1$) respect{i}vely the numerators 
      and denominators of the penult{i}mate convergent in the cont{i}nued fract{i}on expansion of $A/B$. 
      This ensures that $\det N = 1$ always and $\abs{\widetilde{A}} \leq \abs{\widetilde{B}} < \abs{B} = \abs{N}$ except 
      when $\widetilde{B} = 0$ and $A, \widetilde{A}, B \in \F_q^*$ giving $\abs{N} = \abs{B} = 1$ again.
\begin{lem}[\protect{cf.~\citep[Lemma~5]{LN12}}]\label{L:rat}
        Let $k \gg 1$ and $\y \in \Fqtinv^2$ with slope $y = A/B \in \Fqt$. 
        Then, there exists $\gamma = N U (a) M_k \in \SLtwoFqT$ sat{i}sfying
        \[
          \abs{\gamma} = \abs{\frac{y_2}{x_2}}\abs{Q_k Q_{k - 1}} < \abs{\frac{y_2}{x_2}}\abs{Q_k}^2 \text{ and } \abs{\gamma\x - \y} \leq \abs{\frac{B x_2}{Q_k}}.
        \]
\end{lem}
      \begin{proof}
        In this case, $\delta = By - A = 0$ and $\delta' = \widetilde{B}y - \widetilde{A} = 1/B$ so that
        \begin{equation}
          \abs{\Lambda_1 - y \Lambda_2} \leq \frac{\abs{x_2}}{\abs{B Q_k}} 
          \text{ and } \abs{\Lambda_2} \leq \frac{\abs{x_2 B}}{q \abs{Q_k}}.
        \end{equation}
        Af{t}er this, the ultrametric inequality gives the upper bound for \abs{\gamma\x - \y}. Next,
        \begin{equation}
          \abs{a} = \abs{\rho} = \abs{\frac{y_2 Q_k}{x_2 B}} > 1
        \end{equation}
        for all $k$ large enough as $\abs{\eps_k / \eps_{k - 1}}, 
        \abs{\widetilde{B} / B} < 1$ whereas the norm of the f{i}rst term on the right side of~\eqref{E:rho} increases with $k$. 
        We can now improve upon Lemma~\ref{L:gamma} to have more precise knowledge about the size of $\gamma$.
      \end{proof}
\noindent For any such matrix $\gamma$, one thus gets
      \begin{equation}
        \abs{\gamma\x - \y} \leq \abs{\frac{B x_2}{Q_k}} < \abs{B}\abs{x_2 y_2}^{1/2} \abs{\gamma}^{- 1/2}.
      \end{equation}
      We should, however, 
      be more careful when discussing the asymptot{i}c exponent $\mu ( \x, \y )$. Let $\omega < \omega (\xi)$ 
      so that $\abs{Q_{k - 1}} \leq \abs{Q_k}^{1 / \omega}$ for inf{i}nitely many $k$ 
      from the observat{i}ons following~\eqref{E:groQn}. 
      Then, $\abs{\gamma} \ll_{\x, \y} \abs{Q_k}^{1 + \frac{1}{\omega}}$ 
      and $\abs{\gamma\x - \y} \ll \abs{\gamma}^{- \frac{\omega}{\omega + 1}}$ for all such $k$'s. 
      This implies that $\mu ( \x, \y ) \geq \omega / ( \omega + 1 )$ f{i}rstly and 
      since $\omega$ can be taken arbitrarily close to $\omega (\xi)$, we have
      \begin{equation}\label{E:ratmu}
        \mu ( \x, \y ) \geq \frac{\omega (\xi)}{\omega (\xi) + 1}
      \end{equation}
      as in the real case~\citep[\S\S~6.2]{LN12}.\\[-0.1cm]

      Let us move ahead to obtain a similar bound for $\hatmu{\x}{\y}$. 
      Our claim is that it is at least $1 / \big( \omega (\xi) + 1 \big)$. 
      The statement is trivial for $\omega (\xi) = \infty$, thus, we assume that 
      $\omega(\xi)$ is finite. 
      For any $T \gg 1$, there is a unique index $k$ for which
      \begin{equation}
        \abs{\frac{y_2}{x_2}}\abs{Q_{k - 1}Q_k} \leq T < \abs{\frac{y_2}{x_2}}\abs{Q_{k}Q_{k + 1}}.
      \end{equation}
      Then,  $T < \abs{Q_k}^{\omega + 1}$ for given $\omega > \omega (\xi)$ 
      and all $k$ large enough, while Lemma~\ref{L:rat} tells us that 
      there exists a $\gamma = N U(a) M_k \in \SLtwoFqT$ with $\abs{\gamma} \leq T$ 
      and $\abs{\gamma\x - \y} \leq \abs{B x_2} T^{- 1 / ( \omega + 1 )}$. 
      Since this is true for all $T$ large enough and $\omega$ was arbitrary, we conclude that 
      \begin{equation}\label{E:rathatmu}
        \hatmu{\x}{\y} \geq \frac{1}{\omega (\xi) + 1}, 
      \end{equation}
      where $\xi = x_1 / x_2$ is the slope of the start{i}ng point \x.\\[-0.1cm]

Our next goal is to show that the inequali{t}ies~\eqref{E:ratmu} and \eqref{E:rathatmu} are actually equali{t}ies 
when the target \y\ has slope in \Fqt. 
We start with $\mu ( \x, \y )$. 
The trick is to break down any matrix $\gamma$ using the various convergent matrices for $\xi$ and $y$, 
which are already familiar to us from~\eqref{E:conv} (see also~\citep[Theorem~4]{LN12}). 
The theorem given below is an inhomogeneous version of Lemma~\ref{L:lb} for rat{i}onal target points.
\begin{theorem}\label{Th:gap}
Let $\x, \y \in \Fqtinv^2$ be such that $x_1 / x_2 \in \Fqtinv \setminus \Fqt$ and $y_1 / y_2 \in K$.  
For any $k$ such that $\abs{Q_k} > \abs{B x_2 / y_2}$ and $\gamma \in \SLtwoFqT$ 
with $\abs{\gamma} < \abs{y_2 Q_k Q_{k + 1} / x_2}$, 
we must have an approximat{i}on error $\abs{\gamma\x - \y} \geq \abs{x_2 / ( B Q_k )}$.
\end{theorem}
      \begin{proof}
      Let us assume that $\gamma = \begin{pmatrix} V_1 & U_1\\ V_2 & U_2\end{pmatrix}$ is 
        such that $\tp{( \Lambda_1, \Lambda_2 )} = \gamma\x - \y$ sat{i}sf{i}es $\abs{\gamma\x - \y} < \abs{x_2 / ( B Q_k )}$ and we will reach a contradiction. 
        Denote
        \begin{align}
          \widetilde{\gamma} &= \begin{pmatrix}
                              \widetilde{V}_1 & \widetilde{U}_1\\
                              \widetilde{V}_2 & \widetilde{U}_2
                            \end{pmatrix} := N^{-1}\gamma \notag\\
                         &= \begin{pmatrix}
                              \frac{\widetilde{B} ( V_1 y_2 - V_2 y_1 )}{y_2} + \frac{V_2}{B} & \frac{\widetilde{B} ( U_1 y_2 - U_2 y_1 )}{y_2} + \frac{U_2}{B}\\
                              - \frac{B ( V_1 y_2 - V_2 y_1 )}{y_2} & -\frac{B ( U_1 y_2 - U_2 y_1 )}{y_2}
                            \end{pmatrix}
        \end{align}
        so that
        \begin{equation}
          \widetilde{\gamma}\x = x_2 \begin{bmatrix}
                                       \widetilde{V}_1 \xi + \widetilde{U}_1\\
                                       \widetilde{V}_2 \xi + \widetilde{U}_2
                                     \end{bmatrix}
                               = N^{-1} \left( \y + \begin{bmatrix} \Lambda_1\\ \Lambda_2 \end{bmatrix} \right) = \begin{bmatrix}
                                                          \frac{y_2}{B} + \widetilde{B}\Lambda_1 - \widetilde{A}\Lambda_2\\
                                                          -B \Lambda_1 + A\Lambda_2
                                                        \end{bmatrix}
        \end{equation}
        by virtue of the fact that $\widetilde{B}y_1 - \widetilde{A}y_2$ equals $y_2 / B$. Then, the determinant
        \begin{equation}
          \begin{vmatrix} V_1 & y_1\\ V_2 & y_2 \end{vmatrix} =
          \begin{vmatrix}
            V_1 & ( V_1 \xi + U_1 )x_2 - \Lambda_1\\
            V_2 & ( V_2 \xi + U_2 )x_2 - \Lambda_2
          \end{vmatrix} = x_2 - \begin{vmatrix}
                                  V_1 & \Lambda_1\\
                                  V_2 & \Lambda_2
                                \end{vmatrix}
        \end{equation}
        and its norm satisfies 
        \begin{equation}\label{E:Det}
          \abs{V_1 y_2 - V_2 y_1} \leq \max \{ \abs{x_2}, \abs{\gamma}\cdot\abs{\gamma\x - \y} \} \leq \max \{ \abs{x_2}, \abs{y_2 Q_{k + 1} / B} \}.
        \end{equation}
        The second of these terms in the upper bound will clearly dominate when $\abs{Q_k} > \abs{Bx_2 / y_2}$. 
        Because of our supposi{t}ion, one gets for all such $k$ that
        \begin{align}
          \abs{\widetilde{V}_2}                       &= \abs{\frac{B}{y_2} ( V_1 y_2 - V_2 y_1 )} < \abs{Q_{k + 1}}, \text{ and}\notag\\
          \abs{\widetilde{V}_2 \xi + \widetilde{U}_2} &= \frac{1}{\abs{x_2}} \abs{B\Lambda_1 - A \Lambda_2} \leq \abs{\frac{B}{x_2}} \abs{\gamma\x - \y} < \frac{1}{\abs{Q_k}}.
        \end{align}
        We are now ready to consider the matrix $N^{-1}\gamma M_k^{-1} = \widetilde{\gamma}M_k^{-1}$ and more speci{f}ically, 
        tackle its lower right entry given by $\widetilde{V}_2 P_k + Q_k \widetilde{U}_2$. Its size is bounded as follows:
        \begin{equation}\label{E:zero}
          \abs{\widetilde{V}_2 P_k + Q_k \widetilde{U}_2} = \abs{- \widetilde{V}_2 ( Q_k \xi - P_k ) + Q_k ( \widetilde{V}_2 \xi + \widetilde{U}_2 ) } < 1
        \end{equation}
        Since all the three matrices $N, \gamma$ and $M_k$ have polynomial entries and determinant $1$, 
        the entry in~\eqref{E:zero} must be zero. 
        Consequently, $G := N^{-1}\gamma M_k^{-1}$ is of the form
        \begin{equation}
          \begin{pmatrix}
            R & \zeta\\
            -\zeta^{-1} & 0
          \end{pmatrix}
        \end{equation}
        for some $R \in \FqT$ and $\zeta \in \F_q^*$ but that would mean
        \begin{equation}
          \widetilde{\gamma}\x = GM_k \x = x_2 \begin{bmatrix} R\eps_k + ( -1 )^{k - 1} \zeta \eps_{k - 1}\\ - \zeta^{-1} \eps_{k} \end{bmatrix} = \begin{bmatrix}
                            \frac{y_2}{B} + \widetilde{B}\Lambda_1 - \widetilde{A}\Lambda_2\\
                            -B \Lambda_1 + A\Lambda_2
                          \end{bmatrix}.
        \end{equation}
        Let us focus on the second coordinate of this column vector. 
        In the last representat{i}on, 
        $\big|\,B\Lambda_1 - A\Lambda_2\,\big| < \abs{B}\abs{\gamma\x - \y} < \abs{x_2 / Q_k} < \abs{y_2 / B}$ owing to our hypothesis about $k$. Thus, $\widetilde{\gamma}\x$ has norm equal to $\abs{y_2 / B}$ and this in turn forces $\abs{R} = \abs{ y_2 Q_{k + 1} / ( x_2 B )}$. We use this knowledge to get a lower bound for $\abs{\gamma}$ as
        \begin{equation}
          \gamma = NGM_k = \begin{pmatrix}
                             * & *\\
                             BRQ_k + ( -1 )^{k - 1}\zeta BQ_{k - 1} - \zeta^{-1}\widetilde{B}Q_{k} & *
                           \end{pmatrix}
        \end{equation}
        implies that the lef{t}most summand of the lower lef{t} entry rises much faster in size with $k$ 
        than the other two terms whence $\abs{\gamma} \geq \abs{y_2 Q_k Q_{k + 1} / x_2}$, 
        which is a contradict{i}on. The conclusion 
        is that $\abs{\gamma\x - \y}$ has to be at least \abs{x_2 / ( BQ_k )}.
      \end{proof}
Given any $\gamma \in \SLtwoFqT$ of suf{f}iciently large norm, we can f{i}nd a unique $k$ such that
      \begin{equation}
        \abs{\frac{y_2}{x_2}}\abs{Q_{k - 1}Q_k} \leq \abs{\gamma} < \abs{\frac{y_2}{x_2}}\abs{Q_{k}Q_{k + 1}}.
      \end{equation}
      If $\omega (\xi)$ is f{i}nite, 
      choose any $\omega > \omega (\xi)$ so that $\abs{Q_{k - 1}} \geq \abs{Q_k}^{1 / \omega}$ eventually. 
      Then, Theorem~\ref{Th:gap} tells us
      \begin{equation}
        \abs{\gamma\x - \y} \geq \frac{\abs{x_2}}{\abs{BQ_k}} \gg_{\x, \y} \frac{1}{\abs{\gamma}^{\omega / ( \omega + 1 )}}
      \end{equation}
      for all such $\gamma$ and on let{t}ing $\omega$ tend to  $\omega (\xi)$, we have
      \begin{equation}
        \mu ( \x, \y ) \leq \frac{\omega (\xi)}{\omega (\xi) + 1}.
      \end{equation}
When $\omega (\xi) = \infty$, 
we instead argue that $\abs{\gamma\x - \y} \gg_{\x, \y} 1 / \abs{\gamma}$ 
so that $\mu ( \x, \y ) \leq 1 = ( 1 + 1 / \omega (\xi) )^{-1}$ again.\\[-0.1cm]

For the uniform exponent, 
let us f{i}x some $\omega < \omega (\xi)$ 
which means there are inf{i}nitely many denominators $\abs{Q_{k + 1}} \geq \abs{Q_k}^{\omega}$. 
Now, consider the diverging subsequence $H_k := q^{-1} \abs{y_2 Q_k Q_{k + 1} / x_2}$ corresponding to these indices $k$. 
In this scenario, $H_k \geq q^{-1} \abs{y_2 / x_2} \abs{Q_k}^{\omega + 1}$ and Theorem~\ref{Th:gap} is saying
      \begin{equation}
        \abs{\gamma\x - \y} \gg_{\x, \y} \frac{1}{\abs{Q_k}} \gg_{\x, \y} \frac{1}{H_k^{1 / ( \omega + 1 )}}
      \end{equation}
for all $\gamma$ with $\abs{\gamma} \leq H_k$. 
This means that $\hatmu{\x}{\y}$ can be at most $1 / ( \omega + 1 )$ 
and as our choice of $\omega < \omega (\xi)$ was arbitrary, 
we reach  (2) of Theorem \ref{T: exponent for lattice}.

\subsection{Target points with irrat{i}onal slopes}
\label{SSec:irr} 
Let us now start the last case when $\y = \tp{( y_1, y_2 )} \in \Fqtinv^2 \setminus \{ \zero \}$ is 
such that $y = y_1 / y_2 \in \Fqtinv \setminus \Fqt$. 
We further take $\abs{y} \leq 1$ using the matrix $J$. 
This const{i}tutes the generic situat{i}on as far as the target points are concerned. 
If $R_{j - 1} / S_{j - 1}$ and $R_j / S_j$ are any consecut{i}ve cont{i}nued fract{i}on convergents to $y$, 
we take $N_j := \begin{pmatrix} R_j & \widetilde{R}_{j - 1}\\ S_j & \widetilde{S}_{j - 1} \end{pmatrix} \in \SLtwoFqT$, 
where
    \begin{equation}
      \widetilde{R}_{j - 1} = ( -1 )^{j - 1} R_{j - 1} \text{ and } \widetilde{S}_{j - 1} = ( -1 )^{j - 1} S_{j - 1}.
    \end{equation}
Then, $\abs{N_j} = \abs{S_j}$ since that entry dominates all others and the term $\rho$ from~\eqref{E:rho} has absolute value
    \begin{equation}
      \abs{\frac{y_2Q_k}{x_2S_j}} - 1\ < \abs{\rho}\ \leq\ \max \left\{ \abs{\frac{y_2Q_k}{x_2S_j}}, 1 \right\}.
    \end{equation}
The polynomial part $a = [ \rho ]$ has the same norm as $\rho$ whenever $\abs{\rho} \geq 1$ and equals $0$ otherwise. 
    Now, $\abs{\rho} \geq 1$ if and only if 
    $\abs{y_2 Q_k / x_2 S_j} \geq 1$ and therefore, we are right to assert
    \begin{equation}\label{E:a}
      \abs{\frac{y_2Q_k}{x_2S_j}} - 1\ <\ \abs{a}\ \leq\ \abs{\frac{y_2Q_k}{x_2S_j}}.
    \end{equation}
\begin{lem}[\protect{cf.~\cite[Lemma~4]{LN12}}]\label{L:irrjk}
For all $j, k \in \N^*$ and $\x, \y$ as above, 
there exists $\gamma = N_j U (a) M_k \in \SLtwoFqT$ for some $a \in \FqT$ such that
      \[
        \abs{\abs{\frac{y_2 Q_k Q_{k - 1}}{x_2}} - \abs{S_j Q_k}} - \abs{S_{j - 1}Q_{k - 1}}\ 
        \leq\ \abs{\gamma}\ \leq\ \max \big\{ \abs{S_j Q_k}, \abs{\frac{y_2 Q_k Q_{k - 1}}{x_2}} \big\}
      \]
      and
      \[
        \abs{\gamma\x - \y} \leq \max \big\{ \abs{\frac{y_2}{S_j S_{j + 1}}}, \abs{\frac{x_2 S_j}{Q_k}} \big\}.
      \]
    \end{lem}
    \begin{proof}
The quant{i}t{i}es $\delta$ and $\delta'$ introduced in the statement of Lemma~\ref{L:diff} respect{i}vely equal 
$1 / \abs{S_{j + 1}}$ and $1 / \abs{S_j}$ here. 
Also, the second component $\Lambda_2$ of our error vector $\gamma\x - \y$ is bounded above as
      \begin{equation}
        \abs{\Lambda_2} < \abs{\frac{x_2 S_j}{Q_k}}
      \end{equation}
by~\eqref{E:Lambda2}. 
The reasoning for the f{i}rst component is that f{i}rstly,
      \begin{equation}
        \abs{\Lambda_1} \leq \max \{ \abs{\Lambda_1 - y \Lambda_2}, \abs{y\Lambda_2} \} \leq \max \{ \abs{\Lambda_1 - y \Lambda_2}, \abs{\Lambda_2} \}, 
      \end{equation}
while $\abs{\Lambda_1 - y \Lambda_2} \leq \max \{ \abs{y_2}\abs{S_j S_{j + 1}}^{-1}, \abs{x_2}\abs{S_j Q_k}^{-1} \}$ 
from~\eqref{E:first} and \eqref{E:a}. 
Clearly, the term $\abs{x_2 S_j / Q_k}$ will mat{t}er more than $\abs{x_2 / ( S_j Q_k )}$.\\[-0.2cm]

\noindent The bounds ment{i}oned for $\gamma$ follow from Lemma~\ref{L:gamma} and \eqref{E:a}.
    \end{proof}
    \begin{prop}\label{P:irrlb}
      For any pair consist{i}ng of a start{i}ng point $\x \in \Fqtinv^2 \setminus \{ \zero \}$ 
      and target $\y$ with their respect{i}ve slopes $\xi$ and $y$ irrat{i}onal, one has
      \[
        \mu ( \x, \y )\geq\frac{1}{3}\quad\text{and}\quad\hatmu{\x}{\y} \geq \frac{\omega(y)+1}{2(2\omega(y)+1)\omega (\xi)}\geq \frac{1}{4\omega(\xi)}.
      \]
    \end{prop}
\noindent At this stage, 
we follow two parallel strategies as per whether $\omega (\xi) < 3$ or $\omega (\xi) > 2$. 

\subsubsection{The case $\omega (\xi) < 3$} 
For any $j$ large enough, pick $k$ sat{i}sfying
    \begin{equation}\label{E:sw}
      \abs{\frac{y_2Q_{k - 1}}{x_2}}^{1/3} < \abs{S_j} \leq \abs{\frac{y_2Q_{k}}{x_2}}^{1/3} < \abs{S_{j + 1}}.
    \end{equation}
    Let us subst{i}tute this into Lemma~\ref{L:irrjk} which then gives
    \begin{equation}
      \abs{\gamma\x - \y} 
      < \abs{y_2 x_2^2}^{1/3} \max \left\{ \frac{1}{\abs{Q_k Q_{k - 1}}^{1/3}}, \frac{1}{\abs{Q_k}^{2/3}} \right\} 
      = \abs{\frac{y_2 x_2^2}{Q_k Q_{k - 1}}}^{1/3}.
    \end{equation}
    Next, 
    take some $\omega$ with $\omega (\xi) < \omega < 3$ 
    so that $\abs{Q_{k - 1}} \geq \abs{Q_k}^{1 / \omega} > \abs{Q_k}^{1/3}$ 
    for all $k \gg 1$. 
    On the other hand, we also have $\abs{S_j} \ll_{\x, \y} \abs{Q_k}^{1/3}$ by our construct{i}on 
    which dictates that $\abs{\gamma} = \abs{y_2 Q_k Q_{k - 1} / x_2}$. 
    Subsequently, one gets
    \begin{equation}
      \abs{\gamma\x - \y} \ll_{\x, \y} \abs{\gamma}^{- 1/3}, 
    \end{equation}
    for inf{i}nitely many such pairs $( j, k )$ 
    and the inf{i}nite set of matrices $\{ N_j U (a) M_k \} \subset \SLtwoFqT$ determined by them. 
    In part{i}cular, we have shown that $\mu ( \x, \y ) \geq 1/3$.

\subsubsection{\protect{The case $\omega (\xi) > 2$}}
\label{SSSec:geq2} 
F{i}x any $\omega$ with $2 < \omega < \omega (\xi)$. 
Then, there exist inf{i}nitely many $k$ for which $\abs{Q_{k - 1}} \leq \abs{Q_k}^{1 / \omega} < \abs{Q_k}^{1/2}$. 
For each such $k$, we choose $j$ such that
    \begin{equation}\label{E:sw2}
      \abs{S_{j - 1}} < \abs{S_j} \leq \abs{\frac{y_2Q_{k}}{x_2}}^{1/2} < \abs{S_{j + 1}}.
    \end{equation}
The upper right entry $a$ in the unipotent matrix $U (a)$ is chosen to be either the polynomial part $[\rho]$ or $[\rho] + 1$. This will depend on the bot{t}om lef{t} entry $S_j Q_k + ( -1 )^{k - 1} Q_{k - 1} ( S_j a \pm S_{j - 1} )$ having absolute value at least \abs{S_j Q_{k - 1}}\ or not for $a = [ \rho ]$. If
    \begin{equation}
      \abs{S_jQ_k + ( -1 )^{k - 1}Q_{k - 1} ( S_j [\rho] \pm S_{j - 1} )} < \abs{S_j Q_{k - 1}},
    \end{equation}
then it cannot be so for $a = [\rho] + 1$ as well. 
It means that the corresponding matrices $\gamma = N_j U (a) M_k$ will have size at least \abs{S_j Q_{k - 1}}\ 
for one of those choices. 
Both $[\rho]$ and $[\rho] + 1$ have the same size as $\rho$ equal to \abs{y_2 Q_k / S_j}. 
Af{t}er this, 
Lemma~\ref{L:irrjk} provides us an \SLtwoFqT\ matrix $\gamma$ with height 
$\abs{\gamma} \leq \abs{y_2 Q_k^3 / x_2}^{1/2}$ and
    \begin{align}\label{E:1by3}
      \abs{\gamma\x - \y} &\leq \max \{ 1/\abs{S_j}, 1 \} \abs{x_2 y_2}^{1/2} \abs{Q_k}^{ - 1 / 2}\\
      &= \abs{x_2 y_2}^{1/2} \abs{Q_k}^{ - 1 / 2} \ll_{\x, \y} \abs{\gamma}^{-1/3}.\notag
    \end{align}
    The lower bound on $\gamma$ established above ensures that we have inf{i}nitely many \SLtwoFqT\ matrices for which~\eqref{E:1by3} holds and $\mu ( \x, \y ) \geq 1/3$ here too.

\subsubsection{Uniform exponent} 
In this subsect{i}on, we calculate a lower bound for \hatmu{\x}{\y}\ when the target point \y\ has an irrat{i}onal slope $y$. 
Let $\omega (y)$ be the irrat{i}onality measure (with respect to approximat{i}on by elements of \Fqt) of $y$.
\begin{lem}\label{L:irrlb}
If $\tau := \omega(y) / \big( 2\omega(y) + 1 \big),\ \eps > 0$ and $k_0 \gg_{\eps} 1$ is a natural number, 
then there exists a matrix $\gamma \in \SLtwoFqT$ for which
  \[
  \abs{\gamma} \ll_{\x, \y} \abs{Q_k}^2 \text{ and } \abs{\gamma\x - \y} \leq \abs{Q_k}^{\tau - 1 + \eps}.
  \]
\end{lem}
    \begin{proof}
      Observe that $\tau$ is in the range $[ 1/3, 1/2 ]$ as $\omega(y) \in [ 1, \infty ]$. 
      Much like~\cite{LN12}, our choice of the indices $j$ and $k$ in the construct{i}on of $\gamma$ is governed by
      \begin{equation}\label{E:jk}
        \abs{S_j} \leq \abs{Q_k}^{\tau} < \abs{S_{j + 1}}, 
      \end{equation}
      so that both of them go to inf{i}nity together. 
      For $\omega < \infty$, 
      we let $\omega > \omega(y)$ but suf{f}iciently close to make sure 
      that $1/\omega > 1/\omega(y) - \eps/\tau$. 
      Then, $\abs{S_j} \geq \abs{S_{j + 1}}^{1 / \omega}$ for all $j \gg 1$ from \S\,\ref{S:cf} implying 
      that one has $\abs{S_j} \geq \abs{Q_k}^{\tau / \omega(y) - \eps}$ for all large $k$. 
      Now, not{i}ce that this lower bound is trivially true when $\omega(y) = \infty$.\\[-0.2cm]

      With $j$ and $k$ related by~\eqref{E:jk}, Lemma~\ref{L:irrjk} gives us an~\SLtwoFqT\ matrix $\gamma$ such that
      \begin{align}
        \abs{\gamma} &\ll_{\x, \y} \max \{ \abs{S_j Q_k}, \abs{Q_k Q_{k - 1}} \} \leq \abs{Q_k}^2, \text{ and}\\
        \abs{\gamma\x - \y} &\ll_{\x, \y} \max \{ 1/ \abs{S_jS_{j + 1}}, \abs{S_j / Q_k} \} \leq \abs{Q_k}^{\tau - 1 + \eps}.\notag
      \end{align}
      The last dependence on $\x$ and $\y$ is absorbed in the rising $\eps$-powers of \abs{Q_k}'s.
    \end{proof}
Given any $H \gg 1$,
we pick the integer $k$ for which $c_0\abs{Q_k}^2 \leq H < c_0\abs{Q_{k + 1}}^2$ 
where $c_0$ is the implied constant in the upper bound on \abs{\gamma}\ obtained in Lemma~\ref{L:irrlb}. 
When $H$ and consequently $k$ is large enough in terms of $\eps > 0$, 
it follows from the definition of $\omega(\xi)$ 
that $\abs{Q_{k + 1}}^2 \leq \abs{Q_k}^{2\omega(\xi) + \eps}$. 
Hence, we get a matrix $\gamma$ with $\abs{\gamma} \leq H$ and also,
    \begin{equation}
      \abs{\gamma\x - \y} \leq \abs{Q_k}^{- ( 1 - \tau - \eps )} \ll H^{- ( 1 - \tau - \eps ) / (2\omega (\xi) + \eps)}
    \end{equation}
assuming $\eps < 1/2$. In other words, we have
    \begin{equation}
      \hatmu{\x}{\y} \geq \frac{1 - \tau - \eps}{2\omega(\xi) + \eps}
    \end{equation}
where $\eps$ may be erased from the numerator and denominator by taking the limit 
as $\eps $ tends to $0$. 
This gives us the second part of Proposit{i}on~\ref{P:irrlb}.

\subsubsection{A generic upper bound}\label{SSSec:ub}

We now focus on having an upper bound for the asymptot{i}c exponent $\mu ( \x, \y )$. 
This shall be possible for us only for a co-null set consist{i}ng of target vectors 
whose slope has irrat{i}onality measure $1$. 
The lemma below rewrites a matrix $\gamma$ which helps \x\ to reach close to \y\ as a product of the convergent matrices 
$N_j, M_k$ and some $G \in \SLtwoFqT$ for some $j$ and $k$ so that we have good control over the entries of $G$.
 \begin{lem}\label{L:G}
  If $\mu \in [ 0, 1 ]$ and $H \gg_{\x} 1$ are real numbers such that $\gamma$ sat{i}sf{i}es 
  both $\abs{\gamma} \leq qH$ and $\abs{\gamma\x - \y} \leq H^{-\mu}$, 
  $k$ is the index for which $\abs{Q_{k - 1}Q_k} \leq H \leq \abs{Q_kQ_{k + 1}}$ 
  and $j$ is large enough to guarantee $\abs{S_j} \geq H^{\mu / 2}$, 
  then $\gamma = N_j G M_k$ for some \SLtwoFqT\ matrix $G$ 
  whose f{i}rst and second columns are respect{i}vely bounded above in size by
        \[
          c_1 \abs{S_j/Q_k}H^{1 - \mu} \text{ and } c_1\abs{S_jQ_k}H^{-\mu},
        \]
   where $c_1 := \max \{ q, 1/\abs{x_2}, \abs{y_2/x_2}, q/\abs{y_2} \}$.
  \end{lem}
      \begin{proof}
        The matrix $\gamma$ is again taken to be
        \begin{equation}
          \begin{pmatrix}
            V_1 & U_1\\ V_2 & U_2
          \end{pmatrix}
        \end{equation}
        and then, the components $\Lambda_i$ of the di{f}ference $\gamma\x - \y$ are given by~\eqref{E:lambda}. 
        Our hypothesis implies that $\max \{ \abs{\Lambda_1}, \abs{\Lambda_2} \} \leq H^{-\mu}$. 
        From the argument involved in~\eqref{E:Det}, one has $\abs{V_1y_2 - V_2y_1} \leq qH^{1 - \mu}$ and similarly,
        \begin{equation}
          \abs{U_1y_2 - U_2y_1} \leq \max\,\{\,\abs{\xi},\,H^{1 - \mu}\,\} = qH^{1 - \mu}
        \end{equation}
        as we earlier took $\abs{\xi} \leq 1$. The convergent matrix $N_j$ is factored in f{i}rst so that
        \begin{align*}
          \widetilde{\gamma} &=
          \begin{pmatrix}
            \widetilde{V}_1 & \widetilde{U}_1\\
            \widetilde{V}_2 & \widetilde{U}_2
          \end{pmatrix} := N_j^{-1}\gamma =
          \begin{pmatrix}
            \widetilde{S}_{j - 1}V_1 - \widetilde{R}_{j - 1}V_2 & \widetilde{S}_{j - 1}U_1 - \widetilde{R}_{j - 1}U_2\\
            - S_{j}V_1 + R_{j}V_2 & - S_{j}U_1 + R_{j}U_2
          \end{pmatrix}\\
          &= \frac{1}{y_2}
          {\fontsize{6.5pt}{6.5pt}
          \begin{pmatrix}
            \widetilde{S}_{j - 1} ( V_1y_2 - V_2y_1 ) + V_2 ( \widetilde{S}_{j - 1}y_1 - \widetilde{R}_{j - 1}y_2 ) & \widetilde{S}_{j - 1} ( U_1y_2 - U_2y_1 ) + U_2 ( \widetilde{S}_{j - 1}y_1 - \widetilde{R}_{j - 1}y_2 )\\
            -S_{j} ( V_1y_2 - V_2y_1 ) - V_2 ( S_{j}y_1 - R_{j}y_2 ) & -S_{j} ( U_1y_2 - U_2y_1 ) - U_2 ( S_{j}y_1 - R_{j}y_2 )
          \end{pmatrix}}.
        \end{align*}
        Since $\abs{S_jy - R_j} < \abs{\widetilde{S}_{j - 1}y - \widetilde{R}_{j - 1}} = 1/\abs{S_j}$, we get
        \begin{equation}
          \abs{\widetilde{\gamma}} \leq q\cdot\max\,\{\,\abs{S_j/y_2}H^{1 - \mu},\,H/\abs{S_j}\,\} = q\cdot\max \{ \abs{y_2}^{-1}, 1 \} \abs{S_j}H^{1 - \mu}
        \end{equation}
        for all j such that $\abs{S_j} \geq H^{\mu / 2}$. One also has
        \begin{align}
          \widetilde{\gamma}\x &= x_2 \begin{bmatrix}
                                       \widetilde{V}_1 \xi + \widetilde{U}_1\\
                                       \widetilde{V}_2 \xi + \widetilde{U}_2
                                     \end{bmatrix}
                               = N_j^{-1} \left( \y + \begin{bmatrix} \Lambda_1\\ \Lambda_2 \end{bmatrix} \right)\notag\\
          &= \begin{bmatrix}
                                   y_2 ( \widetilde{S}_{j - 1}y - \widetilde{R}_{j - 1} ) + \widetilde{S}_{j - 1}\Lambda_1 - \widetilde{R}_{j - 1}\Lambda_2\\
                                   y_2 ( -y S_j + R_j ) - ( S_j\Lambda_1 - R_j\Lambda_2 )
                                 \end{bmatrix}
        \end{align}
        which implies that $\abs{\widetilde{\gamma}\x} \leq \max \{ \abs{y_2/S_j}, \abs{S_j}H^{-\mu} \} \leq \max \{ \abs{y_2}, 1 \} \abs{S_j}H^{-\mu}$. Let us further append $M_k^{-1}$ to the right of $\widetilde{\gamma}$ and denote
        \begin{align}
          G &:= N_j^{-1}\gamma M_k^{-1} =
          \begin{pmatrix}
            \widetilde{V}_1 & \widetilde{U}_1\\
            \widetilde{V}_2 & \widetilde{U}_2
          \end{pmatrix}
          \begin{pmatrix}
            (-1)^k P_{k - 1} & P_k\\
            (-1)^k Q_{k - 1} & Q_k
          \end{pmatrix}\notag\\
          &= \begin{pmatrix}
               (-1)^k ( P_{k - 1}\widetilde{V}_1 + Q_{k - 1}\widetilde{U}_1 ) & P_k\widetilde{V}_1 + Q_k\widetilde{U}_1\\
               (-1)^k ( P_{k - 1}\widetilde{V}_2 + Q_{k - 1}\widetilde{U}_2 ) & P_k\widetilde{V}_2 + Q_k\widetilde{U}_2
            \end{pmatrix}.
        \end{align}
The top and bot{t}om entries of the lef{t} column of this matrix are the same as
        \begin{equation}
          -\widetilde{V}_i ( Q_{k - 1}\xi - P_{k - 1} ) + Q_{k - 1} ( \widetilde{V}_i\xi + \widetilde{U}_i )
        \end{equation}
for $i$ equal to $1$ and $2$, respect{i}vely and upto mult{i}plicat{i}on by $\pm 1$. 
Both of these expressions are bounded above by
        \begin{equation}
          \max \left\{ \frac{\abs{\widetilde{\gamma}}}{\abs{Q_k}}, \frac{\abs{Q_{k - 1}\abs{\widetilde{\gamma}\x}}}{\abs{x_2}} \right\} 
          \leq \max \left\{ q, \frac{1}{\abs{x_2}}, \abs{\frac{y_2}{x_2}}, \frac{q}{\abs{y_2}} \right\}\,\abs{S_j / Q_k}H^{1 - \mu}.
        \end{equation}
        Insofar as the other column is concerned,
        \begin{align}
          \abs{P_{k}\widetilde{V}_i + Q_{k}\widetilde{U}_i} 
          &= \abs{-\widetilde{V}_i ( Q_k\xi - P_k ) + Q_k ( \widetilde{V}_i\xi + \widetilde{U}_i )}\notag\\
          &\leq \max \{ \abs{\widetilde{\gamma}}/\abs{Q_{k + 1}}, \abs{Q_k}\abs{\widetilde{\gamma}\x}/\abs{x_2} \}\\
          &\leq \max \{ q, 1/\abs{x_2}, \abs{y_2/x_2}, q/\abs{y_2} \}\,\abs{S_jQ_k}H^{-\mu}\notag
        \end{align}
employing the set of inequali{t}ies $\abs{Q_{k - 1}Q_k} \leq H \leq \abs{Q_kQ_{k + 1}}$ to the fullest.
      \end{proof}
Next, we f{i}x $\x \in \Fqtinv^2 \setminus \{ \zero \}, y \in \Fqtinv \setminus \Fqt$ 
such that $\omega(y) = 1$ and $\mu > 1/2$. 
Consider any compact subset $\mathcal{C} \subset \Fqtinv \tp{( y, 1 )} \setminus \{ \zero \}$. 
Following \citeauthor{LN12}~\cite{LN12}, we show that the set
      \begin{equation}
        \mathcal{C}_{\mu} := \{ \y \in \mathcal{C} \mid \abs{\gamma\x - \y} 
        \leq \abs{\gamma}^{-\mu} \text{ for inf{i}nitely many } \gamma \in \SLtwoFqT \}
      \end{equation}
has one-dimensional Lebesgue measure $0$. 
Since $\mathcal{C}$ and $\mu$ were chosen arbitrarily, 
one has that for almost all points $\y \in \Fqtinv\tp{( y, 1 )}$, 
the asymptot{i}c Diophant{i}ne exponent
      \begin{equation}
        \mu ( \x, \y ) \leq 1/2.
      \end{equation}
Suppose $\y \in \mathcal{C}_{\mu}$ and $\abs{\gamma} \gg 1$. 
Our integers $k \geq 1$ and $n \geq 0$ are such that
      \begin{align}\label{E:kn}
        \abs{Q_{k - 1}Q_k} &< \abs{\gamma} \leq \abs{Q_kQ_{k + 1}} \text{ and }\\
       q^{n}\abs{Q_{k - 1}Q_k} &< \abs{\gamma} \leq q^{n + 1}\abs{Q_{k - 1}Q_k}.\notag
      \end{align}
If $H$ is taken to be $q^{n}\abs{Q_{k - 1}Q_k}$, 
we get $\abs{\gamma} \leq qH$ and $\abs{\gamma\x - \y} \leq \abs{\gamma}^{-\mu} < H^{-\mu}$. 
Now, let $j$ be the smallest number for which $\abs{S_j} \geq H^{\mu / 2}$. 
As $\omega (y) = 1$, this denominator $S_j$ cannot be too large in size compared to $S_{j - 1}$. 
Namely, $\abs{S_j} \leq \abs{S_{j - 1}}^{1 + 2\eps/\mu} < H^{\mu/2 + \eps}$ for any $\eps > 0$ 
and all but f{i}nitely many $j$'s depending on $\eps$. 
This dependence has already been accounted for while choosing $\abs{\gamma}$ suf{f}iciently big. 
Subsequently, Lemma~\ref{L:G} informs us of the 
existence of a matrix decomposi{t}ion $\gamma = N_j G M_k$ 
with the f{i}rst column of $G$ bounded by 
$c_2\abs{S_j / Q_k}H^{1 - \mu} \leq c_2 H^{1 - \mu/2 + \eps}/\abs{Q_k}$ 
and the second column by $c_2\abs{Q_k}H^{-\mu/2 + \eps}$. 
The new constant $c_2 = \max_{\y \in \mathcal{C}} \max \{q, 1/\abs{x_2}, 
\abs{\y}/\abs{x_2}, q/\abs{\y} \}$ is a funct{i}on of $\mathcal{C}$ alone for a f{i}xed \x.
\begin{prop}
The number of \SLtwoFqT\ matrices with f{i}rst and second columns (rows) respect{i}vely bounded 
by $B_1$ and $B_2$ is at most $\mathcal{O} ( B_1 B_2 )$.
\end{prop}

      \begin{proof}
        Without loss of generality, we may take $B_1 \leq B_2$ or else interchange their roles. 
        Let us assume that the norm of the f{i}rst column is realized by the entry $f$ in the top row 
        and equals $q^i$ for some $i$ such that $q^i \leq B_1$. 
        Given any such polynomial $f$, 
        the number of allowed entries in the second row of the f{i}rst column is at most $q\Phi(f)$ 
        where $\Phi(f)$ denotes the number of elements in the mult{i}plicat{i}ve group $( \FqT/f\FqT )^*$. 
        With $a_{1,1}$ and $a_{2,1}$ f{i}xed, 
        there is a unique solut{i}on to the equat{i}on 
        $a_{1,2}a_{2,1} \equiv -1$ in $( \FqT/a_{1,1}\FqT )^*$ 
        corresponding to which we have a unique $a_{2,2}$ sat{i}sfying the determinant one condit{i}on. 
        Hence, the number of polynomial vectors \tp{( a_{1,2}, a_{2,2} )}\ 
        such that $a_{1,1}a_{2,2} - a_{1,2}a_{2,1} = 1$ 
        and $\max \{ \abs{a_{1,2}}, \abs{a_{2,2}} \} \leq B_2$ is at most $q B_2 / \abs{f}$. 
        Combining all the choices made, we have our number as
        \begin{equation}
          \ll \sum_{\substack{i \geq 0,\\ q^i \leq B_1}} 
          \sum_{\deg f = i} q\Phi(f)\frac{qB_2}{\abs{f}} \ll_q B_2 \sum_{\substack{i \geq 0,\\ q^i \leq B_1}} 
          \frac{1}{q^i} \sum_{\deg f = i} \Phi(f).
        \end{equation}
        The value of the last inner sum can be obtained 
        from~\citep[Proposit{i}on~2.7]{Ros02} to be $q^{2i} ( q - 1 )$. 
        Thereaf{t}er, the claim is easily seen to be true.
      \end{proof}
For the matrices $G$ which arose in Lemma~\ref{L:G}, 
there are $\mathcal{O} ( c_2^2 H^{1 - \mu + 2\eps} )$ 
possibilit{i}es with $0$ being the obvious upper bound 
if either of $c_2 H^{1 - \mu/2 + \eps}/\abs{Q_k}$ or $c_2\abs{Q_k}H^{-\mu/2 + \eps}$ is less than $1$. 
The ball of diameter $\leq H^{-\mu}$ centered around 
the point $\gamma\x \in \Fqtinv^2$ will not intersect the line 
$\Fqtinv\tp{( y, 1 )}$ in a ball of diameter any bigger than $H^{-\mu}$. 
If we f{i}x $k$ and $n$ in~\eqref{E:kn}, 
we have not more than $\mathcal{O}_{\mathcal{C}, \x} ( H^{1 - \mu + 2\eps} )$ 
such matrices $\gamma \in \SLtwoFqT$. 
Otherwise said, our target point \y\ belongs to some union of balls in $\Fqtinv\tp{( y, 1 )}$ 
whose one-dimensional Lebesgue measure is
      \begin{equation}\label{E:sum}
        \ll_{\mathcal{C}, \x} H^{1 - 2\mu + 2\eps} = ( q^n\abs{Q_{k - 1}Q_k} )^{1 - 2\mu + 2\eps}
      \end{equation}
by the def{i}ni{t}ion of $H$. When $\eps$ is small 
enough so that $1 - 2\mu + 2\eps < 0$ (possible since $\mu > 1/2$),
the term in~\eqref{E:sum} sums up to something f{i}nite 
when we consider all $k \geq 1$ and $n \geq 0$. 
The Borel-Cantelli lemma tells us that the measure 
of the set $\mathcal{C}_{\mu} \subset \Fqtinv\tp{(y, 1)}$ is zero.


\begin{thebibliography}{18} 
  
\bibitem[Bugeaud and Laurent(2005)]{BL05}
Y. Bugeaud and M. Laurent.
On exponents of homogeneous and inhomogeneous {D}iophant{i}ne approximat{i}on.
\emph{Mosc.\ Math.\ J.}, 5(4):747--766, 2005.

\bibitem[Bugeaud and Zhang()]{BZ19}
Y. Bugeaud and Z. L. Zhang.
On homogeneous and inhomogeneous {D}iophant{i}ne approximat{i}on over the f{i}elds of formal power series.
(To appear in Paci{f}ic J.\ Math.).

\bibitem[Cassels(1957)]{Cas57}
J.~W.~S. Cassels.
\emph{An introduction to {D}iophant{i}ne approximat{i}on}.
Cambridge Tracts in Mathemat{i}cs and Mathemat{i}cal Physics, No. 45.
Cambridge University Press, New York, 1957.

\bibitem[Fuchs(2010)]{Fuc10}
M. Fuchs.
Metrical theorems for inhomogeneous {D}iophant{i}ne approximat{i}on in posit{i}ve characterist{i}c.
\emph{Acta Arith.}, 141(2): 191--208, 2010.

\bibitem[{Ghosh} et~al.(){Ghosh}, {Gorodnik}, and {Nevo}]{GGN3}
A. Ghosh, A. Gorodnik and A. Nevo.
Best possible rates of distribut{i}on of dense lat{t}{i}ce orbits in homogeneous spaces.
(To appear in J.\ Reine Angew.\ Math.).

\bibitem[Ghosh et~al.(2015)Ghosh, Gorodnik, and Nevo]{GGN15}
A. Ghosh, A. Gorodnik and A. Nevo.
Diophant{i}ne approximat{i}on exponents on homogeneous variet{i}es.
In \emph{Recent trends in ergodic theory and dynamical systems},
\emph{Contemp.\ Math.}
631: 181--200, 2015

\bibitem[Khinchin(1964)]{Khi64}
A.~Ya. Khinchin.
\emph{Cont{i}nued fract{i}ons}.
The University of Chicago Press, Chicago, Ill.-London, 1964.

\bibitem[Kim and Nakada(2011)]{KN11}
D. H. Kim and H. Nakada.
Metric inhomogeneous {D}iophant{i}ne approximat{i}on on the f{i}eld of formal {L}aurent series.
\emph{Acta Arith.}, 150(2): 129--142, 2011.

\bibitem[Kristensen(2011)]{Kri11}
S.~Kristensen.
Metric inhomogeneous {D}iophant{i}ne approximat{i}on in posi{t}ive characterist{i}c.
\emph{Math.\ Scand.}, 108(1): 55--76, 2011.

\bibitem[Laurent and Nogueira(2012)]{LN12}
M. Laurent and A. Nogueira.
Approximat{i}on to points in the plane by $\mathrm{SL} (2,\mathbb{Z})$-orbits.
\emph{J.\ London Math.\ Soc.}, 85(2): 409--429, 2012.

\bibitem[Ma and Su(2008)]{MS08}
C. Ma and W. Y. Su.
Inhomogeneous {D}iophant{i}ne approximat{i}on over the f{i}eld of formal {L}aurent series.
\emph{F{i}nite F{i}elds Appl.}, 14(2): 361--378, 2008.

\bibitem[Maucourant and Weiss(2012)]{MW12}
F. Maucourant and B. Weiss.
Lat{t}{i}ce act{i}ons on the plane revisited.
\emph{Geom. Dedicata}, 157: 1--21, 2012.

\bibitem[Pollicot{t}(2011)]{Pol11}
M. Pollicot{t}.
Rates of convergence for linear act{i}ons of cocompact lat{t}{i}ces on the complex plane.
\emph{Integers}, 11B: Paper No.~A12, 2011.

\bibitem[Rosen(2002)]{Ros02}
M. Rosen.
\emph{Number theory in funct{i}on f{i}elds}, Springer-Verlag, New York, 2002.

\bibitem[Schmidt(2000)]{Sch00}
W. M. Schmidt.
On cont{i}nued fract{i}ons and diophant{i}ne approximat{i}on in power series fields.
\emph{Acta Arith.}, 95(2): 139--166, 2000.

\bibitem[Singhal(2017)]{Sin17}
L.~Singhal.
Diophant{i}ne exponents for standard linear act{i}ons of {${\rm SL}_2$} over discrete rings in {$\mathbb{C}$}.
\emph{Acta Arith.}, 177(1): 53--73, 2017.

\bibitem[Sprindzhuk(1969)]{Spr69}
V.~G. Sprindzhuk.
\emph{Mahler's problem in metric number theory}.
American Mathematical Soc., 1969.

\bibitem[Zheng()]{Zhe17}
Z. Y. Zheng.
Simultaneous {D}iophant{i}ne {A}pproximat{i}on in {F}unct{i}on {F}ields.
arXiv preprint arXiv:1711.03721, 2017.
  
\end{thebibliography}
\end{document}